\documentclass{amsart}

\usepackage{mathrsfs}
\usepackage{stmaryrd}
\usepackage{enumerate}
\usepackage{tikz-cd}
\usepackage[all]{xy}
\usepackage{aliascnt}
\usepackage[colorlinks=true, linkcolor=blue, citecolor=blue]{hyperref}

\newtheorem{theorem}{Theorem}

\newaliascnt{lemma}{theorem}
\newtheorem{lemma}[lemma]{Lemma}
\aliascntresetthe{lemma}

\newaliascnt{corollary}{theorem}
\newtheorem{corollary}[corollary]{Corollary}
\aliascntresetthe{corollary}

\newaliascnt{proposition}{theorem}
\newtheorem{proposition}[proposition]{Proposition}
\aliascntresetthe{proposition}

\newaliascnt{conjecture}{theorem}

\aliascntresetthe{conjecture}

\newaliascnt{question}{theorem}
\newtheorem{question}[question]{Question}
\aliascntresetthe{question}

\theoremstyle{definition}

\newaliascnt{definition}{theorem}
\newtheorem{definition}[definition]{Definition}
\aliascntresetthe{definition}

\newaliascnt{remark}{theorem}
\newtheorem{remark}[remark]{Remark}
\aliascntresetthe{remark}

\newaliascnt{example}{theorem}

\aliascntresetthe{example}

\newaliascnt{notation}{theorem}

\aliascntresetthe{notation}

\newtheorem*{acknowledgements}{Acknowledgements}

\newcommand{\Z}{\mathbb{Z}}

\newcommand{\cX}{\mathcal{X}}
\newcommand{\cY}{\mathcal{Y}}
\newcommand{\cZ}{\mathcal{Z}}
\newcommand{\cW}{\mathcal{W}}
\newcommand{\cI}{\mathcal{I}}
\newcommand{\cC}{\mathcal{C}}
\newcommand{\cD}{\mathcal{D}}
\newcommand{\cU}{\mathcal{U}}
\newcommand{\cO}{\mathcal{O}}
\newcommand{\cV}{\mathcal{V}}
\newcommand{\cF}{\mathcal{F}}
\newcommand{\cG}{\mathcal{G}}
\newcommand{\cH}{\mathcal{H}}

\newcommand{\cM}{\mathcal{M}}

\newcommand{\sL}{\mathscr{L}}
\newcommand{\sM}{\mathscr{M}}
\newcommand{\sI}{\mathscr{I}}
\newcommand{\sJ}{\mathscr{J}}

\newcommand{\bL}{\mathbb{L}}
\newcommand{\bH}{\mathbb{H}}
\newcommand{\bG}{\mathbb{G}}
\newcommand{\bA}{\mathbb{A}}

\newcommand{\fm}{\mathfrak{m}}

\newcommand{\Stack}{\mathbf{Stack}}

\newcommand{\diff}{\mathrm{d}}
\newcommand{\red}{\mathrm{red}}
\newcommand{\id}{\mathrm{id}}
\newcommand{\Gor}{\mathrm{Gor}}

\DeclareMathOperator{\het}{ht}
\DeclareMathOperator{\e}{e}
\DeclareMathOperator{\Spec}{Spec}
\DeclareMathOperator{\Ext}{Ext}
\DeclareMathOperator{\ord}{ord}
\DeclareMathOperator{\Hom}{Hom}
\DeclareMathOperator{\Aut}{Aut}
\DeclareMathOperator{\GL}{GL}
\DeclareMathOperator{\len}{len}
\DeclareMathOperator{\tors}{tors}
\DeclareMathOperator{\Tor}{Tor}
\DeclareMathOperator{\coh}{coh}

\tikzset{cong/.style={draw=none,edge node={node [sloped, allow upside down, auto=false]{$\cong$}}},
         Isom/.style={above,every to/.append style={edge node={node [sloped, allow upside down, auto=false]{$\sim$}}}}}

\title{A motivic change of variables formula for Artin stacks}

\author{Matthew Satriano and Jeremy Usatine}

\thanks{MS was partially supported by a Discovery Grant from the National Science and Engineering Research Council of Canada.}

\address{Matthew Satriano, Department of Pure Mathematics, University of Waterloo}
\email{msatriano@uwaterloo.ca}

\address{Jeremy Usatine, Department of Mathematics, Brown University}
\email{jeremy{\_}usatine@brown.edu}

\begin{document}

\begin{abstract}
Let $\mathcal{X} \to Y$ be a birational map from a smooth Artin stack to a (possibly singular) variety. We prove a change of variables formula that relates motivic integrals over arcs of $Y$ to motivic integrals over arcs of $\mathcal{X}$. With a view toward the study of stringy Hodge numbers, this change of variables formula leads to a notion of crepantness for the map $\mathcal{X} \to Y$ that coincides with the usual notion in the special case that $\cX$ is a scheme.
\end{abstract}

\maketitle

\numberwithin{theorem}{section}
\numberwithin{lemma}{section}
\numberwithin{corollary}{section}
\numberwithin{proposition}{section}
\numberwithin{conjecture}{section}
\numberwithin{question}{section}
\numberwithin{remark}{section}
\numberwithin{definition}{section}
\numberwithin{example}{section}
\numberwithin{notation}{section}

\setcounter{tocdepth}{1}
\tableofcontents

\section{Introduction}

Let $Y$ be a variety. The theory of motivic integration, introduced by Kontsevich \cite{Kontsevich}, defines a \emph{motivic measure} on the \emph{arc scheme} of $Y$ that takes values in a certain completed Grothendieck ring of varieties. A key feature of this motivic measure is a \emph{change of variables formula} that, given a birational modification $X \to Y$, relates the motivic measure for $Y$ to the motivic measure for $X$, see e.g., \cite{DenefLoeser1999, Looijenga}. This change of variables formula is an essential ingredient in applications of motivic integration to many areas, including birational geometry, mirror symmetry, and the study of singularities, especially when one wants to understand invariants of a singular variety in terms of some resolution of singularities.

%Various work (too numerous to adequately survey here) 
Numerous results across the last half century has suggested that smooth stacks are a natural tool in studying singular varieties: this perspective goes back at least as far as Deligne and Mumford's seminal work on moduli spaces of curves \cite{DeligneMumford}, is essential in recent work on functorial resolution of singularities \cite{AbramovichTemkinWlodarczyk, McQuillanMarzo}, and pervades (at the very least implicitly) work on the McKay correspondence since Reid's formulation \cite{Reid}. In an example that is particularly relevant to this paper, Yasuda proved that if $Y$ is proper and has Gorenstein quotient singularities and $\cX \to Y$ is the canonical smooth Deligne--Mumford stack associated to $Y$, then the stringy Hodge numbers of $Y$ coincide with the orbifold Hodge numbers of $\cX$ \cite{Yasuda2004}. This result was a consequence of Yasuda's motivic integration for Deligne--Mumford stacks and a change of variables formula \cite[Theorem 1.3]{Yasuda2004} for the map $\cX \to Y$, showing that $\cX$ behaves like a \emph{crepant} resolution of $Y$ (see \autoref{IntroSubsectionCrepantness} below). In particular, this work showed that these stringy Hodge numbers are nonnegative, confirming a special case of a conjecture of Batyrev's \cite[Conjecture 3.10]{Batyrev1998} and demonstrating the utility of a motivic integration and change of variables formula for Deligne--Mumford stacks (even when \emph{a priori} one only cares about singular varieties).

A natural class of varieties with far worse than quotient singularities is those varieties that are the good moduli space $Y$, in the sense of \cite{Alper}, of a smooth Artin stack $\cX$. Such varieties naturally arise in the context of geometric invariant theory: they include all toric varieties as well as quotients of affine space by representations of linearly reductive groups. They generally have worse than quotient singularities and therefore have no obvious \emph{crepant} resolution by a smooth Deligne--Mumford stack. However, the good moduli space map $\cX \to Y$ is a natural candidate for a nice resolution of singularities by a smooth \emph{Artin} stack. In \cite{SatrianoUsatine}, the authors introduced a conjectural (proven in the toric case) framework that uses motivic integration for the stack $\cX$ to study the stringy Hodge numbers of the singular variety $Y$. This framework suggested that when $\cX \to Y$ is birational and has exceptional locus with codimension at least 2, there should be a motivic change of variables formula making $\cX$ behave like a crepant resolution of $Y$ (see \cite[Conjecture 1.1 and Theorem 1.7]{SatrianoUsatine} for precise details). In particular this illuminated that, purely from the perspective of studying singular varieties, there is need for
\begin{enumerate}

\item\label{introItemsChangeOfVariables} a versatile motivic change of variables formula for maps $\cX \to Y$ from smooth Artin stacks to singular varieties and 

\item\label{introItemsCrepantness} a notion of \emph{crepantness} for such maps $\cX \to Y$.\footnote{It is important to note that unlike the case of Deligne-Mumford stacks, defining the canonical bundle for Artin stacks is a subtle issue, so there is no \emph{a priori} obvious definition one can take for $\cX\to Y$ to be crepant.}

\end{enumerate}

In this paper we prove a motivic change of variables formula, \autoref{theoremChangeOfVariables}, that applies in a high level of generality to birational maps from smooth Artin stacks to singular varieties, addressing (\ref{introItemsChangeOfVariables}). We also discuss how the form this change of variables formula takes suggests a useful notion for crepantness, addressing (\ref{introItemsCrepantness}).

\begin{remark}
Since \cite{Yasuda2004}, Yasuda has further developed the theory of motivic integration for \emph{Deligne--Mumford} stacks and obtained quite general change of variables formulas in that setting, see e.g., \cite{Yasuda2006, Yasuda2019}. 

Also, Balwe introduced a motivic integration and change of variables formula for Artin $n$-stacks \cite[Theorem 7.2.5]{Balwe}. However, his hypothesis that the map be ``0-truncated'' prevents his formula from applying to our main case of interest here, namely maps form Artin stacks to varieties.
\end{remark}

\begin{remark}
Although we are not aware of any direct concrete relationship, this work shares a spiritual connection with the literature on variation of geometric invariant theory in the context of the D-equivalence conjecture, such as \cite{BallardFaveroKatzarkov, DonovanSegal, HalpernLeistner, HalpernLeistnerSam}, in which derived categories of Artin stacks are used to produce derived equivalences between certain birational varieties. In a vague sense, this connection is analogous to the relationship between the motivic McKay correspondence and the derived McKay correspondence.
\end{remark}

\subsection{Conventions}

Throughout this paper, let $k$ be an algebraically closed field of characteristic 0. For any stack $\cX$ over $k$, we let $|\cX|$ denote its associated topological space, and for any subset $\cC \subset |\cX|$ and field extension $k'$ of $k$, we let $\cC(k')$ (resp. $\overline{\cC}(k')$) denote the category of (resp. set of isomorphism classes of) $k'$-valued points of $\cX$ whose class in $|\cX|$ is contained in $\cC$.

\subsection{Change of variables formula} 

The main result of this paper is a change of variables formula relating motivic integrals over arcs of a possibly-singular variety $Y$ to motivic integrals over arcs of a smooth Artin stack $\cX$. A key input to this change of variables formula is a certain \emph{height function} $\het^{(0)}_{E}$ associated to an object $E$ in the derived category of $\cX$. The value of $\het^{(0)}_{E}$ at an arc $\Spec(k'\llbracket t \rrbracket) \to \cX$ is given by pulling back $E$ to $\Spec(k'\llbracket t \rrbracket)$ and taking the dimension (over $k'$) of the resulting complex's 0th (hyper)cohomology, see \autoref{definitionHeightFunctions} for a precise definition of $\het^{(0)}_{E}$. We expect that these height functions will play an important role in motivic integration for Artin stacks, even beyond the change of variables formula, as \autoref{SumOfHeightFunctionIsMeasurable} shows they are natural functions to integrate. 

We now state our change of variables formula. Our notions of cylinders, the arc stack $\sL(\cX)$, the motivic measure $\mu_\cX$, and constructible functions and their integrals are all straightforward generalizations of the corresponding notions in the case of schemes. See \autoref{sectionMotivicIntegrationForSmoothArtinStacks} or \cite[Section 3]{SatrianoUsatine} for precise definitions.

\begin{theorem}\label{theoremChangeOfVariables}
Let $\cX$ be a smooth irreducible finite type Artin stack over $k$ with affine geometric stabilizers and separated diagonal, let $Y$ be an irreducible finite type scheme over $k$ with $\dim Y = \dim \cX$, let $\cX \to Y$ be a morphism, and let $\cU$ be an open substack of $\cX$ such that $\cU \hookrightarrow \cX \to Y$ is an open immersion.

\begin{enumerate}[(a)]

\item\label{theoremMeasurablePartOfCOV} Let $L_{\cX/Y}$ and $L_{\cI_\cX/ \cX}$ be the relative contangent complexes of $\cX \to Y$ and the inertia map $\cI_\cX \to \cX$, respectively, and let $\sigma: \cX \to \cI_\cX$ be the identity section of $\cI_\cX \to \cX$. For any cylinder $\cC \subset |\sL(\cX)| \setminus |\sL(\cX \setminus \cU)| \subset |\sL(\cX)|$, the function 
\[
	\het^{(0)}_{L\sigma^*L_{\cI_\cX/\cX}} - \het^{(0)}_{L_{\cX/Y}} : \cC \to \Z
\]
is constructible.

\item\label{theoremPartChangeOfVariables} Let $\cC \subset |\sL(\cX)| \setminus |\sL(\cX \setminus \cU)| \subset |\sL(\cX)|$ and $D \subset \sL(Y)$ be cylinders such that, for all field extensions $k'$ of $k$, the map $\overline{\cC}(k') \to D(k')$ is a bijection. Then
\[
	\mu_Y(D) = \int_{\cC}\bL^{\het^{(0)}_{L\sigma^*L_{\cI_\cX/\cX}} - \het^{(0)}_{L_{\cX/Y}}} \diff \mu_\cX.
\]

\end{enumerate}
\end{theorem}

\begin{remark}
By \autoref{heightFunctionRemarkWhenNotInfinity} below, $\het^{(0)}_{L\sigma^*L_{\cI_\cX/\cX}}$ and $\het^{(0)}_{L_{\cX/Y}}$ take only integer values on $\cC$.
\end{remark}

We see that the height functions $\het^{(0)}_{L_{\cX/Y}}$ and $\het^{(0)}_{L\sigma^*L_{\cI_\cX/\cX}}$ control the change of variables formula. Note that in the case where $\cX$ is a scheme, the function $\het^{(0)}_{L_{\cX/Y}}$ coincides with the order function of the relative jacobian ideal of $\cX \to Y$, i.e., is the function controlling Kontsevich's original motivic change of variables formula. However, when $\cX$ is a scheme (or even a Deligne--Mumford stack), $L\sigma^*L_{\cI_\cX/\cX}$ is isomorphic to the 0 object, so the function $\het^{(0)}_{L\sigma^*L_{\cI_\cX/\cX}}$ is identically zero. Therefore $\het^{(0)}_{L\sigma^*L_{\cI_\cX/\cX}}$ is a novel contribution that only appears when dealing with Artin stacks that are not Deligne--Mumford. Note that this function, which controls the stabilizer groups of the jet stacks of $\cX$ (see \autoref{theoremStabilizersOfJets}), is nonzero in general (see e.g., \cite[Section 5]{SatrianoUsatine}, or consider the stabilizers of jets in simple cases such as $[\bA_{k}^2/ \bG_{m,k}]$ where $\bG_{m,k}$ acts with weights 1 and -1).

Finally, we note that cylinders $\cC$ and $D$ satisfying the hypotheses of \autoref{theoremChangeOfVariables} seem to arise frequently. For example, the change of variables formula  \cite[Theorem 1.7]{SatrianoUsatine} for toric Artin stacks is easily reduced to cylinders satisfying these hypotheses by \cite[Theorem 4.9]{SatrianoUsatine}. Another family of examples of such $\cC$ and $D$ is the subject of \cite[Proposition 10.2]{SatrianoUsatine}, which are examples involving quotients of affine space by certain representations of $\mathrm{SL}_2$.

\begin{remark}
\cite[Theorem 1.7]{SatrianoUsatine} does not immediately follow from \autoref{theoremChangeOfVariables}, as such an implication would require additional control of $\het^{(0)}_{L\sigma^*L_{\cI_\cX/\cX}} - \het^{(0)}_{L_{\cX/Y}}$ in that case, see \autoref{questionSmallResolutionCrepant} below. 
\end{remark}

\subsection{Crepantness and applications to stringy invariants}\label{IntroSubsectionCrepantness}

Now suppose $Y$ is a variety over $k$ with log-terminal singularities, and for expositional simplicity, also assume the canonical divisor $K_Y$ is Cartier. Batyrev introduced a notion of \emph{stringy Hodge numbers} for $Y$ \cite{Batyrev1998}. If $Y$ is proper and $X \to Y$ is a crepant resolution of $Y$ (by a scheme $X$), the stringy Hodge numbers of $Y$ coincide with the usual Hodge numbers of $X$. Refining the notion of stringy Hodge numbers, Denef and Loeser introduced the \emph{Gorenstein measure} $\mu^\Gor_Y$ on $\sL(Y)$ \cite{DenefLoeser2002}. If $X \to Y$ is a crepant resolution of $Y$ (by a scheme $X$), then $\mu^\Gor_Y$ is essentially the usual motivic measure $\mu_X$. A key insight of Yasuda's was that if $Y$ has only quotient singularities and $\cX \to Y$ is the canonical smooth Deligne--Mumford stack associated to $Y$, then $\mu^\Gor_Y$ is essentially (an orbifold version, in the sense of orbifold cohomology \cite{ChenRuan}, of) $\mu_\cX$, i.e., from this perspective $\cX$ behaves like a crepant resolution of $Y$.

Now suppose that $\cX$ is an Artin stack over $k$, that $\cX \to Y$ is a morphism, and that these satisfy the hypotheses of \autoref{theoremChangeOfVariables}. Unlike when $\cX$ is a scheme (or even a Deligne--Mumford stack), we are not aware of a general definition of relative canonical divisor $K_{\cX/Y}$ that produces a nice definition for crepantness of $\cX \to Y$. We therefore take an approach hinted at by the properties of $\mu^\Gor_Y$ above: we propose that one should consider $\cX \to Y$ to be crepant when $\mu^\Gor_Y$ and $\mu_\cX$ have a certain nice relationship; specifically, we propose that $\cX \to Y$ should be considered crepant if whenever $\cC$ and $D$ satisfy the hypotheses in \autoref{theoremChangeOfVariables}, we have $\mu^\Gor_Y(D) = \mu_\cX(\cC)$. \autoref{theoremChangeOfVariables} tells us how to characterize this property in terms of the height functions $\het^{(0)}_{L\sigma^*L_{\cI_\cX/\cX}}$ and $\het^{(0)}_{L_{\cX/Y}}$, as we now make precise.

Let $\omega_{Y} = \iota_*\Omega^{\dim Y}_{Y_\mathrm{sm}}$, where $\iota: Y_\mathrm{sm} \to Y$ is the inclusion of the smooth locus, and let $\sJ_Y$ be the unique ideal sheaf on $Y$ such that $\sJ_Y \omega_Y$ is the image of $\Omega_Y^{\dim Y} \to \omega_Y$. Recall that for any cylinder $D \subset \sL(Y)$,
\[
	\mu^\Gor_Y(D) = \int_D \bL^{\ord_{\sJ_Y}} \diff \mu_Y.
\]

\begin{definition}\label{definitionCrepantness}
Let $\cX$ be a smooth irreducible finite type Artin stack over $k$ with affine geometric stabilizers and separated diagonal, let $Y$ be an integral finite type separated scheme over $k$ with log-terminal singularities and Cartier canonical divisor, assume $\dim Y = \dim \cX$, and let $\pi: \cX \to Y$ be a morphism. We say $\pi$ satisfies ($\triangle$) if there exists a nonempty open substack $\cU$ of $\cX$ such that $\cU \hookrightarrow \cX \to Y$ is an open immersion and
\[
	\ord_{\sJ_Y} \circ \sL(\pi) + \het^{(0)}_{L\sigma^*L_{\cI_\cX/\cX}} - \het^{(0)}_{L_{\cX/Y}} = 0
\]
on $|\sL(\cX)| \setminus |\sL(\cX \setminus \cU)|$.
\end{definition}

\begin{remark}
If $\cX = X$ is a variety, $(\triangle)$ coincides with the condition $K_{X/Y} = 0$. Indeed, when $\cX = X$ is a variety
\[
	\ord_{\sJ_Y} \circ \sL(\pi) + \het^{(0)}_{L\sigma^*L_{\cI_\cX/\cX}} - \het^{(0)}_{L_{\cX/Y}} = \ord_{\sJ_Y} \circ \sL(\pi) - \mathrm{ordjac}_\pi,
\]
and the latter vanishes if and only if $K_{X/Y} = 0$, see e.g., \cite[Chapter 7 Proposition 3.2.5]{ChambertLoirNicaiseSebag}.
\end{remark}

An immediate consequence of \autoref{theoremChangeOfVariables} and \autoref{definitionCrepantness} is that if $\pi$ satisfies $(\triangle)$, then $\mu^\Gor_Y(D) = \mu_\cX(\cC)$ for any $\cC, D$ satisfying the hypotheses of \autoref{theoremChangeOfVariables}(\ref{theoremPartChangeOfVariables}). We therefore propose $(\triangle)$ as a notion for crepantness. We expect a better understanding of this condition will have applications to studying stringy Hodge numbers. As an example of such an application, we end this introduction by posing the following natural question, which explicitly ties \autoref{theoremChangeOfVariables} and \autoref{definitionCrepantness} back to the framework proposed in \cite{SatrianoUsatine}.

\begin{question}\label{questionSmallResolutionCrepant}
If $\pi: \cX \to Y$ is a small resolution, i.e., there exists an open subset $U \subset Y$ such that $\pi^{-1}(U) \to U$ is an isomorphism and $\cX \setminus \pi^{-1}(U)$ has codimension at least 2, does $\pi$ satisfy $(\triangle)$?
\end{question}

Note that \autoref{theoremChangeOfVariables} shows that \cite[Conjecture 1.1]{SatrianoUsatine} is essentially equivalent to \autoref{questionSmallResolutionCrepant} having an affirmative answer.

\begin{acknowledgements}
It is a pleasure to thank Dan Abramovich and Dan Edidin for helpful conversations. We would especially like to thank Ariyan Javanpeykar and Siddharth Mathur for helpful discussions concerning \autoref{stratifyLemmaIntoGaGerbes}.
\end{acknowledgements}

\section{Motivic integration for smooth Artin stacks}\label{sectionMotivicIntegrationForSmoothArtinStacks}

In this section, we set up the formalism and notation for a motivic integration for smooth Artin stacks. For any Artin stack $\cX$ over $k$, we let $\sL(\cX)$ denote the arc stack of $\cX$, and for each $n \in \Z_{\geq 0}$, we let $\sL_n(\cX)$ denote the $n$th jet stack of $\cX$. In other words, $\sL(\cX) = \varprojlim_{n} \sL_n(\cX)$ and $\sL_n(\cX)$ is the Weil restriction of $\cX \otimes_k k[t]/(t^{n+1})$ along the morphism $\Spec(k[t]/(t^{n+1})) \to \Spec(k)$. We will let $\theta_n: \sL(\cX) \to \sL_n(\cX)$ and $\theta^n_m: \sL_n(
\cX) \to \sL_m(\cX)$ denote the truncation morphisms. See \cite[Section 3]{SatrianoUsatine} for some of the basic properties of $\sL(\cX)$ and $\sL_n(\cX)$. In particular, we recall that if $\cX$ is finite type over $k$, then so is each $\sL_n(\cX)$, and if $\cX$ has affine geometric stabilizers, so does each $\sL_n(\cX)$. For us, the important subsets of $|\sL(\cX)|$ will be the so-called cylinders, whose definition we now recall.

\begin{definition}
Let $\cX$ be a finite type Artin stack over $k$. A subset of $|\sL(\cX)|$ is called a \emph{cylinder} if it is of the form $\theta_n^{-1}(\cC_n)$ for some $n \in \Z_{\geq 0}$ and constructible subset $\cC_n \subset |\sL_n(\cX)|$.
\end{definition}

\begin{proposition}\label{propositionConstructibleTopologyIsCompact}
Let $\cX$ be a finite type Artin stack over $k$. Then any cover of $|\sL(\cX)|$ by cylinders has a finite subcover.
\end{proposition}

\begin{proof}
Let $V \to \cX$ be a finite type smooth cover. Then infinitesimal lifting implies that $\sL(V) \to |\sL(\cX)|$ is surjective. It is easy to verify that the preimage in $\sL(V)$ of a cylinder in $|\sL(\cX)|$ is a cylinder, so the proposition reduces to the well known case where $\cX$ is a scheme, see e.g., \cite[Chapter 6 Example 4.1.3]{ChambertLoirNicaiseSebag}. 
\end{proof}

For the remainder of this section, fix a smooth irreducible finite type Artin stack $\cX$ over $k$ with affine geometric stabilizers. We will recall that there is a motivic measure $\mu_\cX$ that assigns to each cylinder in $|\sL(\cX)|$ an element of a modified Grothendieck ring of varieties. First, we set some notation for that modified Grothendieck ring of varieties. We refer to \cite[Subsection 2.2]{SatrianoUsatine} for further details. Let $\widehat{\sM}_k$ denote the ring obtained by starting with the Grothendieck ring of $k$-varieties, inverting the class of $\bA_k^1$, and then completing with respect to the dimension filtration. Let $K_0(\Stack_k)$ denote the Grothendieck ring of stacks in the sense of \cite{Ekedahl}, and let $K_0(\Stack_k) \to \widehat{\sM}_k$ denote the unique ring homomorphism that takes each class in $K_0(\Stack_k)$ of any finite type $k$-scheme to its class in $\widehat{\sM}_k$. For any finite type Artin stack $\cY$ over $k$ with affine geometric stabilizers, let $\e(\cY) \in \widehat{\sM}_k$ denote the image along $K_0(\Stack_k) \to \widehat{\sM}_k$ of the class of $\cY$ in $K_0(\Stack_k)$, and for any constructible subset $\cD \subset |\cY|$, let $\e(\cD) \in \widehat{\sM}_k$ denote the image along $K_0(\Stack_k) \to \widehat{\sM}_k$ of the class of $\cD$ in $K_0(\Stack_k)$. We will let $\bL$ denote the class of $\bA_k^1$ in $\widehat{\sM}_k$. We now recall the definition of $\mu_\cX$.

\begin{definition}
Let $\cC \subset |\sL(\cX)|$ be a cylinder. The \emph{motivic measure} of $\cC$ is
\[
	\mu_\cX(\cC) = \varinjlim_n \e(\theta_n(\cC))\bL^{-(n+1)\dim\cX} \in \widehat{\sM}_k.
\]
Note that each $\theta_n(\cC) \subset |\sL_n(\cX)|$ is constructible and the above limit converges (in fact the sequence stabilizes) by \cite[Theorem 3.33]{SatrianoUsatine}.
\end{definition}

\begin{definition}
Let $\cC \subset |\sL(\cX)|$ be a cylinder. A function $f: \cC \to \Z$ is called \emph{constructible} if for all $n \in \Z$, the set $f^{-1}(n) \subset |\sL(\cX)|$ is a cylinder.
\end{definition}

\begin{remark}
We use the terminology ``constructible function'' for consistency with \cite{ChambertLoirNicaiseSebag}. However, unlike in the case where $\cX$ is a scheme, we do not know if cylinders coincide with constructible subsets of the topological space $|\sL(\cX)|$ (using the general definition of constructible subset for not-necessarily-noetherian spaces).
\end{remark}

\begin{remark}\label{remarkOrdIsConstructible}
Let $\sI$ be a quasi-coherent ideal sheaf on $\cX$. We get a function $\ord_\sI: |\sL(\cX)| \to \Z \cup \{\infty\}$ in the usual manner: for any $\varphi: \Spec(k'\llbracket t \rrbracket) \to \cX$, we set $\ord_\sI(\varphi)$ to be such that $\sI \cdot k'\llbracket t \rrbracket = (t^{\ord_{\sI}(\varphi)})$, where $(t^\infty)$ means $(0)$. Then just as in the case where $\cX$ is a scheme, it is easy to verify that for any $n \in \Z$ the set $\ord_\sI^{-1}(n)$ is a cylinder. Thus for any cylinder $\cC \subset |\sL(\cX)| \setminus |\sL(\cY)|$, where $\cY$ is the closed substack defined by $\sI$, we have $\ord_\sI: \cC \to \Z$ is constructible.
\end{remark}

\begin{remark}\label{sumOfMeasurableFunctionsIsMeasurable}
If $f,g: |\sL(\cX)| \to \Z$ are constructible functions, then the function $f+g: |\sL(\cX)| \to \Z$ is constructible by \autoref{propositionConstructibleTopologyIsCompact}.
\end{remark}

\begin{definition}
Let $\cC \subset |\sL(\cX)|$ be a cylinder, and let $f: \cC \to \Z$ be a constructible function. We define the \emph{integral} of $\bL^f$ to be
\[
	\int_{\cC} \bL^{f} \diff\mu_\cX = \sum_{n \in \Z} \mu_\cX(f^{-1}(n))\bL^{n},
\]
where we note that this sum is finite by \autoref{propositionConstructibleTopologyIsCompact}.
\end{definition}

\section{Height functions for arcs and jets}

In this section, fix a finite type Artin stack $\cX$ over $k$ and an object $E\in D^-_{\coh}(\cX)$ in the derived category of $\cX$.

\begin{definition}\label{definitionHeightFunctions}
For any $i \in \Z$, any field extension $k'$ of $k$, and any arc $\varphi: \Spec(k'\llbracket t \rrbracket) \to \cX$ (resp. jet $\varphi_n: \Spec(k'[t]/(t^{n+1})) \to \cX$), set
\[
	\het_E^{(i)}(\varphi) = \dim_{k'} \bH^i (L\varphi^*E) \quad (\text{resp. } \het_{n,E}^{(i)}(\varphi_n) = \dim_{k'} \bH^i (L\varphi_n^*E) ).
\]
The assignment $\varphi \mapsto \het_E^{(i)}(\varphi)$ (resp. $\varphi_n \mapsto \het_{n,E}^{(i)}(\varphi_n)$) induces functions
\[
	\het_E^{(i)}: |\sL(\cX)| \to \Z_{\geq 0} \cup \{\infty\} \quad (\text{resp. } \het_{n,E}^{(i)}: |\sL_n(\cX)| \to \Z_{\geq 0} ).
\]
that we refer to as the associated \emph{height functions}.
\end{definition}

The main result of this section is the following theorem, which collects some important properties of these height functions.

\begin{theorem}\label{theoremMainPropertiesOfHeightFunctions}
Let $i \in \Z$ and $n \in \Z_{\geq 0}$.
\begin{enumerate}[(a)]

\item\label{theoremSemicontinuityOfHeightFunction} The function $\het_{n,E}^{(i)}: |\sL_n(\cX)| \to \Z_{\geq 0}$ is upper semi-continuous.

\item\label{theoremRelatingArcHeightFunctionToJetHeightFunction} For any $\varphi \in |\sL(\cX)|$,
\[
	\het_E^{(i)}(\varphi) + \het_E^{(i+1)}(\varphi) \geq \het_{n,E}^{(i)}(\theta_n(\varphi)),
\]
and if this inequality is strict, then
\[
	\het_{n,E}^{(i)}(\theta_n(\varphi)) \geq n+1 \qquad \text{or} \qquad \het^{(i+1)}_E(\varphi) = \infty.
\]
In particular, if $n+1 \geq \het_E^{(i)}(\varphi) + \het_E^{(i+1)}(\varphi)$, then
\[
	\het_E^{(i)}(\varphi) + \het_E^{(i+1)}(\varphi) = \het_{n,E}^{(i)}(\theta_n(\varphi)).
\]
\end{enumerate}
\end{theorem}

We prove \autoref{theoremMainPropertiesOfHeightFunctions} after a preliminary lemma.

\begin{lemma}\label{lenMfornsufflarge}
Let $k'$ be a field extension of $k$, let $R=k'\llbracket t \rrbracket$, and let $M$ be a finitely generated $R$-module. Then for all $n\geq0$,
\[
\dim_{k'}(M)\geq \dim_{k'}(M\otimes_R R/(t^{n+1})) \geq \dim_{k'}\Tor_1^R(M,R/(t^{n+1})),
\]
and the second inequality is an equality if $\dim_{k'}(M) \neq \infty$. Furthermore, if 
\begin{align*}
	&\dim_{k'}(M) > \dim_{k'}(M\otimes_R R/(t^{n+1}))\\
	&\text{ (resp. $\dim_{k'}(M) > \dim_{k'}\Tor_1^R(M,R/(t^{n+1}))$)},
\end{align*}
then 
\begin{align*}
	&\dim_{k'}(M\otimes_R R/(t^{n+1})) \geq n+1\\
	&\text{ (resp. $\dim_{k'}\Tor_1^R(M,R/(t^{n+1})) \geq n+1$ or $\dim_{k'}M = \infty$)}.
\end{align*}
\end{lemma}
\begin{proof}
By the structure theorem for modules over a PID, $M\cong R^r\oplus\bigoplus_{i=1}^m R/(t^{e_i})$ for $e_i\geq1$. Since $\Tor^R_1(R, R/(t^{n+1})) = 0$, we reduce immediately to the case where $M=R/(t^e)$ for some $e \geq 1$. Then from the short exact sequence
\[
0\to R\stackrel{\cdot t^e}{\to} R\to M\to 0,
\]
we obtain a long exact sequence
\[
0\to \Tor_1^R(M,R/(t^{n+1}))\to R/(t^{n+1})\stackrel{\cdot t^e}{\to} R/(t^{n+1})\to M\otimes_R R/(t^{n+1})\to 0,
\]
and so $\dim_k(M\otimes_R R/(t^{n+1}))=\dim_k\Tor_1^R(M,R/(t^{n+1}))$. Since $M\otimes_R R/(t^{n+1})=R/(t^e,t^{n+1})=R/(t^{\min(e,n+1)})$, we see $\dim_{k'}(M)\geq \dim_{k'}(M\otimes_R R/(t^{n+1}))$, which finishes the proof of the first statement. Furthermore, if $e = \dim_{k'}M > \dim_{k'}(M\otimes_R R/(t^{n+1})) = \min(e, n+1)$, we have $\dim_{k'}(M\otimes_R R/(t^{n+1})) = \min(e,n+1) = n+1$, and the remainder of the lemma follows.
\end{proof}

\begin{proof}[{Proof of \autoref{theoremMainPropertiesOfHeightFunctions}}]
We first prove part (\ref{theoremSemicontinuityOfHeightFunction}). Let $V \to \sL_n(\cX)$ be a finite type smooth cover by a scheme, and consider the composition
\[
	\widetilde{\het}: V \to |\sL_n(\cX)| \xrightarrow{\het_{n,E}^{(i)}} \Z_{\geq 0}.
\]
Because $V \to |\sL_n(\cX)|$ is surjective and open, to show (\ref{theoremSemicontinuityOfHeightFunction}) it is sufficient to prove that $\widetilde{\het}$ is upper semi-continous. The map $V \to \sL_n(\cX)$ corresponds to a map
\[
	\Phi: V \times_k \Spec(k[t]/(t^{n+1})) \to \cX.
\]
Let $\pi: V \times_k \Spec(k[t]/(t^{n+1})) \to V$ be the projection, and for any $v \in V$, let $Y_v$ be the fiber of $\pi$ over $v$ and let $\iota_v: Y_v \to V \times_k \Spec(k[t]/(t^{n+1}))$ be the inclusion. Then for all $v \in V$,
\[
	\widetilde{\het}(v) = \dim_{k(v)}\bH^i(L\iota_v^*L\Phi^*E),
\]
where $k(v)$ is the residue field of $v$. Set $F = R\pi_* L\Phi^*E$. Because $\pi$ is flat and proper, cohomology and base change (see e.g., \cite[Lemma 0CSC]{stacks-project}) implies that $F$ is pseudo-coherent (i.e.~$F\in \in D^-_{\coh}(V)$) and that for all $v \in V$,
\[
	\bH^i(L\iota_v^*L\Phi^*E) \cong H^i(F \otimes_{\cO_V}^{L} k(v)).
\]
Part (\ref{theoremSemicontinuityOfHeightFunction}) then follows from the fact that $v \mapsto \dim_{k(v)}H^i(F \otimes_{\cO_V}^{L} k(v))$ is upper semi-continous, see e.g., \cite[Lemma 0BDI]{stacks-project}.

We now prove part (\ref{theoremRelatingArcHeightFunctionToJetHeightFunction}). Let $R = k'\llbracket t \rrbracket$, where $k'$ is a field extension of $k$ over which $\varphi$ is defined. We have a spectral sequence
\[
E_2^{pq}= \Tor^R_{-p}(L^q\varphi^* E, R/(t^{n+1}))\Rightarrow L^{p+q}\varphi_n^*E
\]
see, e.g., \cite[Equation (3.10)]{Huybrechts2006}. Since $R$ is a PID, all $E_2^{pq}$ terms vanish except for $p\in\{-1,0\}$. Thus,
\begin{align*}
	\het_{n,E}^{(i)}(\theta_n(\varphi)) &= \dim_{k'}L^{i}\varphi_n^*E\\
	 &= \dim_{k'}(L^i\varphi^* E \otimes_R R/(t^{n+1})) + \dim_{k'}(\Tor^R_1(L^{i+1}\varphi^* E, R/(t^{n+1})),
\end{align*}
and we are done by \autoref{lenMfornsufflarge}. 
\end{proof}

The following corollary gives important examples of constructible functions.

\begin{corollary}\label{SumOfHeightFunctionIsMeasurable}
Let $i \in \Z$, let $f: |\sL(\cX)| \to \Z \cup \{\infty\}$ be the sum of height functions $\het_E^{(i)} + \het_E^{(i+1)}$, and let $\cC \subset |\sL(\cX)|$ be a cylinder such that $\cC \cap (\het^{(i+1)}_E)^{-1}(\infty) = \emptyset$. Then for any $n \in \Z$, the set $f^{-1}(n) \cap \cC$ is a cylinder.
\end{corollary}

\begin{proof}
By \autoref{theoremMainPropertiesOfHeightFunctions}(\ref{theoremSemicontinuityOfHeightFunction}), the set $(\het_{n-1,E}^{(i)} \circ \theta_{n-1})^{-1}(\Z_{\geq n}) \subset |\sL(\cX)|$ is a cylinder, so it is sufficient to prove that
\[
	f^{-1}(\Z_{\geq n}) \cap \cC= (\het_{n-1,E}^{(i)} \circ \theta_{n-1})^{-1}(\Z_{\geq n}) \cap \cC,
\]
but this follows easily from \autoref{theoremMainPropertiesOfHeightFunctions}(\ref{theoremRelatingArcHeightFunctionToJetHeightFunction}).
\end{proof}

\begin{remark}\label{heightFunctionRemarkWhenNotInfinity}
If $\cU$ is an open substack of $\cX$ such that $E$ restricts to 0 on $\cU$, then each $\het_E^{(i)}$ is never $\infty$ on $|\sL(\cX)| \setminus |\sL(\cX \setminus \cU)|$. Thus \autoref{SumOfHeightFunctionIsMeasurable} implies that $\het_E^{(i)} + \het_E^{(i+1)}: \cC \to \Z$ is a constructible function for any cylinder $\cC \subset |\sL(\cX)| \setminus |\sL(\cX \setminus \cU)|$.
\end{remark}

We may now prove \autoref{theoremChangeOfVariables}(\ref{theoremMeasurablePartOfCOV}).

\begin{proof}[Proof of \autoref{theoremChangeOfVariables}(\ref{theoremMeasurablePartOfCOV})]
Using the notation of \autoref{theoremChangeOfVariables}, we have $L_{\cX/Y}$ and $L\sigma^*L_{\cI_\cX/\cX}$ live in $D^-_{\coh}(\cX)$. Because $H^iL_{\cX/Y} = 0$ for $i > 1$ and $H^i L\sigma^*L_{\cI_\cX/\cX}= 0$ for $i > 0$ (because $\cI_\cX \to \cX$ is representable), we have that  $\het_{L_{\cX/Y}}^{(i)}$ is identically 0 for $i > 1$ and $\het_{L\sigma^*L_{\cI_\cX/\cX}}^{(i)}$ is identically 0 for $i > 0$. Thus, also noting that $L_{\cX/Y}$ and $L\sigma^*L_{\cI_\cX/\cX}$ restrict to 0 on $\cU$, \autoref{heightFunctionRemarkWhenNotInfinity} implies that $\het_{L_{\cX/Y}}^{(0)} + \het_{L_{\cX/Y}}^{(1)}$, $\het_{L_{\cX/Y}}^{(1)}$, and $\het_{L\sigma^*L_{\cI_\cX/\cX}}^{(0)}$ are constructible functions on $\cC$. We are done by \autoref{sumOfMeasurableFunctionsIsMeasurable}.
\end{proof}

We end this section by proving the following lemma, which will allow us to relate our height functions to the various deformation spaces that will appear throughout the proof of \autoref{theoremChangeOfVariables}(\ref{theoremPartChangeOfVariables}).

\begin{lemma}\label{replacingExtWithH}
Let $i \in \Z$, let $n \geq m \in \Z_{\geq 0}$, let $k'$ be a field, let $\iota: \Spec(k'[t]/(t^m)) \to \Spec(k'[t]/(t^n))$ be the truncation map, and let $F$ be a psuedo-coherent object in the derived category of $\Spec(k'[t]/(t^{n}))$. Then
\[
	\dim_{k'}\Ext^i_{k'[t]/(t^{n})}(F, (t^n)/(t^{n+m})) = \dim_{k'} \bH^{-i}(L\iota^*F).
\]
\end{lemma}

\begin{proof}
Let $F^\bullet$ be a complex of projectives representing the object $F$. Then
\begin{align*}
	\Ext^i_{k'[t]/(t^{n})}(F, &(t^n)/(t^{n+m})) \cong H^i \Hom_{k'[t]/(t^{n})}( F^\bullet, (t^n)/(t^{n+m}))\\
	&\cong H^i\Hom_{k'[t]/(t^{n})}( F^\bullet, k'[t]/(t^m) )\\
	&\cong H^i\Hom_{k'[t]/(t^{m})}( F^\bullet \otimes_{k'[t]/(t^n)} k'[t]/(t^m), k'[t]/(t^m) ),
\end{align*}
where the second isomorphism is due to the fact that $k'[t]/(t^m) \cong (t^n)/(t^{n+m})$ as $k'[t]/(t^{n})$-modules. Baer's criterion and the fact that $k'\llbracket t \rrbracket$ is a principal ideal domain implies that $k'[t]/(t^m)$ is an injective $k'[t]/(t^m)$-module, so
\begin{align*}
	H^i\Hom_{k'[t]/(t^{m})}&( F^\bullet \otimes_{k'[t]/(t^n)} k'[t]/(t^m), k'[t]/(t^m) )\\
	&\cong \Hom_{k'[t]/(t^{m})}( H^{-i}(F^\bullet \otimes_{k'[t]/(t^n)} k'[t]/(t^m)), k'[t]/(t^m) )\\
	&\cong H^{-i}(F^\bullet \otimes_{k'[t]/(t^n)} k'[t]/(t^m))\\
	&\cong L^{-i} \iota^* F,
\end{align*}
where the second isomorphism is due to the fact that  $\Hom_{k'[t]/(t^{m})}( M, k'[t]/(t^m) ) \cong M$ for any finitely generated $k'[t]/(t^{m})$-module $M$ as a straightforward consequence of the structure theorem for finitely generated modules over $k'\llbracket t \rrbracket$. The lemma follows from the above isomorphisms and the fact that $\Spec(k'[t]/(t^m))$ is affine.
\end{proof}

\section{Stabilizers of jet stacks}

Throughout this section, let $\cX$ be a finite type Artin stack over $k$ such that the inertia map $\cI_\cX \to \cX$ is separated, and let $\sigma: \cX \to \cI_\cX$ denote the identity section of $\cI_\cX \to \cX$. We will use the fact that since $\cI_\cX \to \cX$ is separated, $\sigma$ is a closed immersion.

\begin{theorem}\label{theoremStabilizersOfJets}
Let $\sJ$ denote the ideal sheaf on $\cI_\cX$ defining the closed substack $\cX \xrightarrow{\sigma} \cI_\cX$.
\begin{enumerate}[(a)]

\item\label{existenceOfIdealControllingStabilizers} If $\cU$ is an open substack of $\cX$ that is isomorphic to an algebraic space, then there exists a quasi-coherent ideal sheaf $\sI$ on $\cX$ such that
\begin{itemize}

\item $\sI\cO_{\cI_\cX} \cdot \sJ = 0$, and

\item the closed substack of $\cX$ defined by $\sI$ has underlying topological space equal to $|\cX| \setminus |\cU|$.

\end{itemize}

\item\label{computationOfStabliizersOfJets} Let $\sI$ be a quasi-coherent ideal sheaf on $\cX$ such that $\sI \cO_{\cI_\cX} \cdot \sJ = 0$, let $k'$ be a field extension of $k$, and let $\varphi \in \sL(\cX)(k')$. Then for all 
\[
	n \geq \max(\ord_\sI(\varphi), 2\max(\ord_{\sI}(\varphi), \het^{(0)}_{L\sigma^*L_{\cI/\cX}}(\varphi)) - 1),
\]
the stabilizer of $\theta_n(\varphi) \in \sL_n(\cX)(k')$ is isomorphic as an algebraic group to
\[
	\bG_{a,k'}^{\het^{(0)}_{L\sigma^*L_{\cI_\cX/\cX}}(\varphi)}.
\]

\end{enumerate}
\end{theorem}

We begin with a lemma that will be used to prove \autoref{theoremStabilizersOfJets}(\ref{existenceOfIdealControllingStabilizers}).

\begin{lemma}
\label{ClosedSubstackIsoOverComplementOfOtherClosedSubstackIdealLemma}
Let $\cY$ be a noetherian Artin stack, and let $\cZ, \cW$ be closed substacks of $\cY$. Let $\sI$ and $\sJ$ be the ideal sheaves on $\cY$ defining $\cZ$ and $\cW$, respectively, and let $\cU \hookrightarrow \cY$ be the open substack such that $|\cU| = |\cY| \setminus |\cZ|$. If $\cW \times_\cY \cU \to \cU$ is an isomorphism, then there exists some $r \in \Z_{>0}$ such that
\[
	\sI^r \cdot \sJ = 0.
\]
\end{lemma}

\begin{proof}
For any $r \in \Z_{>0}$, let $\cV_r \hookrightarrow \cY$ denote the closed substack defined by $\sI^r \cdot \sJ$. It is sufficient to prove that there exists some $r \in \Z_{>0}$ such that $\cV_r \hookrightarrow \cY$ is an isomorphism. Let $\Spec(A) \to \cY$ be a smooth cover. Then it is sufficient to show that there exists some $r \in \Z_{>0}$ such that $\cV_r \times_{\cY} \Spec(A) \hookrightarrow \Spec(A)$ is an isomorphism. Let $I, J \subset A$ be the global sections of the ideal sheaves $\sI \cO_{\Spec(A)}, \sJ \cO_{\Spec(A)}$, respectively. Because $\cV_r \times_{\cY} \Spec(A) \hookrightarrow \Spec(A)$ is the closed subscheme defined by $I^r \cdot J$, it is sufficient to prove that there exists some $r \in \Z_{>0}$ such that $I^r \cdot J = 0$. Because $A$ is noetherian, it is sufficient to show that for any $f \in I$, every element of $J$ is $f$-torsion.

Let $f \in I$, and let $A_f$ be the ring obtained from $A$ by inverting $f$. Then the open subscheme $\Spec(A_f) \hookrightarrow \Spec(A)$ is contained in the open subscheme $\cU \times_\cY \Spec(A) \hookrightarrow \Spec(A)$. Set $B = A_f \otimes_A (A/J)$.  Then $\Spec(B) \to \Spec(A_f)$ is the base change of $\cW \to \cY$ along the composition $\Spec(A_f) \hookrightarrow \cU \times_{\cY} \Spec(A) \to \cU \hookrightarrow \cY$, so $A_f \to B$ is an isomorphism. Therefore
\[
	J = \ker(A \to A/J) \subset \ker(A \to B) = \ker(A \to A_f),
\]
so every element of $J$ is $f$-torsion, and the proof is complete.
\end{proof}

We now prove \autoref{theoremStabilizersOfJets}(\ref{existenceOfIdealControllingStabilizers}).

\begin{proof}[Proof of \autoref{theoremStabilizersOfJets}(\ref{existenceOfIdealControllingStabilizers})]
Let $\sI'$ be the ideal sheaf on $\cX$ that defines the reduced substack with underlying topological space equal to $|\cX| \setminus |\cU|$, and let $\cZ \hookrightarrow \cI_\cX$ be the closed substack defined by $\sI' \cO_{\cI_\cX}$. Then $|\cI_\cU| = |\cI_\cX| \setminus |\cZ|$. Let $\cU \to \cI_\cU \to \cU$ be the base change of $\cX \xrightarrow{\sigma} \cI_\cX \to \cX$ along $\cU \hookrightarrow \cX$. Then $\cU \to \cI_\cU$ is a section of the isomorphism $\cI_\cU \xrightarrow{\sim} \cU$ (this is an isomorphism because $\cU$ is an algebraic space), so $\cU \to \cI_\cU$ is an isomorphism. Thus $\cX \xrightarrow{\sigma} \cI_\cX$ is a closed immersion such that $\cX \times_{\cI_{\cX}} \cI_\cU \to \cI_\cU$ is an isomorphism. Therefore \autoref{ClosedSubstackIsoOverComplementOfOtherClosedSubstackIdealLemma} implies that there exists some $r \in \Z_{>0}$ such that
\[
	(\sI' \cO_{\cI_\cX})^r \cdot \sJ = 0.
\]	
Setting $\sI = (\sI')^r$, we get
\[
	\sI \cO_{\cI_\cX} \cdot \sJ = 0,
\]
and the closed substack of $\cX$ defined by $\sI$ has underlying topological space equal to $|\cX| \setminus |\cU|$.
\end{proof}

The next lemmas will be used in proving \autoref{theoremStabilizersOfJets}(\ref{computationOfStabliizersOfJets}).

\begin{lemma}
\label{lemmaAboutFlatnessAndTorsionAndTruncations}
Let $k'$ be a field, let $n \in \Z_{\geq 0}$, let $A$ be a flat $k'[t]/(t^{n+1})$ algebra, and let $a \in A$. If $m \leq n$ is such that $t^ma = 0$, then $a \in (t^{n-m+1})A$.
\end{lemma}

\begin{proof}
Consider the exact sequence
\[
	0 \to (t^{n-m+1})/(t^{n+1}) \to k'[t]/(t^{n+1}) \xrightarrow{\cdot t^m} k'[t]/(t^{n+1}).
\]
Tensoring with the flat algebra $A$, we get the sequence
\[
	0 \to (t^{n-m+1})/(t^{n+1}) \otimes_{k'[t]/(t^{n+1})} A \to A \xrightarrow{\cdot t^m} A,
\]
whose exactness implies the lemma.
\end{proof}

\begin{lemma}
\label{lemmaAboutFactoringMapThroughClosedSubschemeAfterTruncating}
Let $k'$ be a field, let $Y$ be a scheme over $k'\llbracket t \rrbracket$, let $W \hookrightarrow Y$ be a closed subscheme, let $n \in \Z_{\geq 0}$, let $Z$ be a flat scheme over $k'[t]/(t^{n+1})$, and let $Z \to Y$ be a morphism over $k'\llbracket t \rrbracket$. Let $\sJ$ be the ideal sheaf on $Y$ that defines $W$. If $m \leq n$ is such that $(t^m) \cO_Y \cdot \sJ = 0$, then $Z \times_{\Spec(k'[t]/(t^{n+1}))} \Spec(k'[t]/(t^{n-m+1})) \to Z \to Y$ factors through $W \hookrightarrow Y$.
\end{lemma}

\begin{proof}
It is sufficient to prove that if $\Spec(B) \subset Y$ is an affine open and $\Spec(A) \subset Z$ is an affine open mapped into $\Spec(B)$ by $Z \to Y$, then $\Spec(A / (t^{n-m+1})) \to \Spec(A) \to \Spec(B)$ factors through $W \times_Y \Spec(B) \hookrightarrow \Spec(B)$.

Let $\Spec(B) \subset Y$ be an affine open, let $\Spec(A) \subset Z$ be an affine open mapped into $\Spec(B)$ by $Z \to Y$, and let $J \subset B$ be the set of global sections of $\sJ \cO_{\Spec(B)}$, so the closed immersion $W \times_Y \Spec(B) \hookrightarrow \Spec(B)$ is defined by the ideal $J$, and
\[
	(t^m)J = 0.
\]
Let $b \in J$, and let $a \in A$ be the image of $b$ under $B \to A$. Then $t^ma = 0$, so by \autoref{lemmaAboutFlatnessAndTorsionAndTruncations}, we have $a \in (t^{n-m+1})A$. Thus $B \to A/(t^{n-m+1})$ factors through $B \to B/J$, and we are done.
\end{proof}

For the remainder of this section, let $\sJ$ denote the ideal sheaf on $\cI_\cX$ defining the closed substack $\cX \xrightarrow{\sigma} \cI_\cX$. We will eventually prove \autoref{theoremStabilizersOfJets}(\ref{computationOfStabliizersOfJets}) by expressing the relevant stabilizers as certain deformation spaces. The following proposition is key to setting up the appropriate deformation problem.

\begin{proposition}
\label{autosEventuallyTruncateToSection}
Let $k'$ be a field extension of $k$, let $\varphi: \Spec(k'\llbracket t \rrbracket) \to \cX$ be a morphism over $k$, let $n \in \Z_{\geq 0}$, let $A$ be a $k'$-algebra, and let $g_n: \Spec(A[t]/(t^{n+1})) \to \cI_\cX$ make the diagram
\begin{center}
\begin{tikzcd}
\Spec(A[t]/(t^{n+1})) \arrow{r}{g_n} \arrow[d] & \cI_\cX \arrow[d] \\
\Spec(k'\llbracket t \rrbracket) \arrow{r}{\varphi} & \cX
\end{tikzcd}
\end{center}
2-commute. If $\sI$ is an ideal sheaf on $\cX$ such that $\sI \cO_{\cI_\cX} \cdot \sJ = 0$ and $n \geq \ord_{\sI}(\varphi)$, then the diagram
\begin{center}
\begin{tikzcd}
&\Spec(k'\llbracket t \rrbracket) \arrow{r}{\varphi} & \cX \arrow{d}{\sigma}\\
\Spec(A[t]/(t^{n-\ord_\sI(\varphi)+1})) \arrow[r] \arrow[ru] &\Spec(A[t]/(t^{n+1})) \arrow{r}{g_n}  & \cI_\cX 
\end{tikzcd}
\end{center}
2-commutes.
\end{proposition}

\begin{proof}
Let $\Spec(k'\llbracket t \rrbracket) \xrightarrow{s} G \to \Spec(k' \llbracket t \rrbracket)$ be the base change of $\cX \xrightarrow{\sigma} \cI_\cX \to \cX$ along $\Spec(k'\llbracket t \rrbracket) \xrightarrow{\varphi} \cX$. By choosing a 2-isomorphism that makes
\begin{center}
\begin{tikzcd}
\Spec(A[t]/(t^{n+1})) \arrow{r}{g_n} \arrow[d] & \cI_\cX \arrow[d] \\
\Spec(k'\llbracket t \rrbracket) \arrow{r}{\varphi} & \cX
\end{tikzcd}
\end{center}
2-commute, we obtain a map $h_n: \Spec(A[t]/(t^{n+1})) \to G$ making
\begin{center}
\begin{tikzcd}
\Spec(A[t]/(t^{n+1})) \arrow[bend left]{drr}{g_n} \arrow[bend right]{ddr} \arrow{dr}{h_n} &&\\
& G \arrow[d] \arrow[r] & \cI_\cX \\
&\Spec(k'\llbracket t \rrbracket)& 
\end{tikzcd}
\end{center}
2-commute. Set $m = \ord_\sI(\varphi)$. Let $h_{n-m}: \Spec(A[t]/(t^{n-m+1})) \to G$ be the composition of $\Spec(A[t]/(t^{n-m+1})) \hookrightarrow \Spec(A[t]/(t^{n+1}))$ with $\Spec(A[t]/(t^{n+1})) \xrightarrow{h_n} G$. It is sufficient to prove that the map $h_{n-m}$ factors through the closed immersion $\Spec(k'\llbracket t \rrbracket) \xrightarrow{s} G$.

The closed immersion $\Spec(k'\llbracket t \rrbracket) \xrightarrow{s} G$ is defined by the ideal sheaf $\sJ \cO_G$, and
\[
	(t^m) \cO_G \cdot \sJ \cO_G = (\sI \cO_{\Spec(k'\llbracket t \rrbracket)}) \cO_G \cdot \sJ\cO_G = (\sI\cO_{\cI_\cX} \cdot \sJ)\cO_G = 0.
\]
Therefore \autoref{lemmaAboutFactoringMapThroughClosedSubschemeAfterTruncating} implies that $h_{n-m}$ factors through the closed immersion $\Spec(k'\llbracket t \rrbracket) \xrightarrow{s} G$ as desired.
\end{proof}

Before completing the proof of \autoref{theoremStabilizersOfJets}(\ref{computationOfStabliizersOfJets}), we need the following lemma. We expect this is well-known, but we include a proof for lack of a suitable reference.

\begin{lemma}
\label{LieAlgArgument}
Let $S$ be a scheme and $G$ a (not necessarily flat) group scheme over $S$ with multiplication map $m\colon G\times_S G\to G$ and identity section $e\colon S\to G$. Then for all $\cO_S$-modules $\cF$, the map
\[
\Hom(e^*\Omega^1_{G/S},\cF)\times\Hom(e^*\Omega^1_{G/S},\cF)\to\Hom(e^*\Omega^1_{G/S},\cF)
\]
induced by $m$ is given by addition of functions $(\varphi_1,\varphi_2)\mapsto \varphi_1+\varphi_2$.
\end{lemma}
\begin{proof}
%If $p_j\colon G\times_S G\to G$ denotes the projection onto the $j$-th factor, we have $\Omega^1_{(G\times_SG)/G}\simeq p_1^*\Omega^1_{G/S}\oplus p_2^*\Omega^1_{G/S}$. If needed, can add more details: the way it works is you take $p_j^*\Omega^1_{G/S}\to\Omega^1_{(G\times_SG)/G}$ and pull back by $i_j$ to get the identity map which shows that the $\pi_j$ are what we claim they are.
We have an isomorphism
\[
(e\times e)^*\Omega^1_{(G\times_S G)/S}\simeq e^*\Omega^1_{G/S}\oplus e^*\Omega^1_{G/S}
\]
where the two projection maps $\pi_j\colon(e\times e)^*\Omega^1_{(G\times_S G)/S}\to e^*\Omega^1_{G/S}$ can be described as follows. Let
\[
i_1,i_2\colon G\to G\times_S G
\]
be given by
\[
i_1(g)=(g,e)\quad\textrm{and}\quad i_2(g)=(e,g).
\]
Since $e\times e=i_j\circ e$ for $j\in\{1,2\}$, applying $e^*$ to the morphisms $i_j^*\Omega^1_{G/S}\to\Omega^1_{(G\times_S G)/S}$ yield two maps $(e\times e)^*\Omega^1_{(G\times_S G)/S}\to e^*\Omega^1_{G/S}$, which are precisely the maps $\pi_j$.

Next, $m$ yields a map $\mu\colon m^*\Omega^1_{G/S}\to\Omega^1_{(G\times_S G)/S}$. Since $m\circ i_j=\id$, it follows that $\pi_j\circ(e\times e)^*\mu=\id$. In other words, the map 
\[
e^*\Omega^1_{G/S}\xrightarrow{(e\times e)^*\mu} (e\times e)^*\Omega^1_{(G\times_S G)/S}\simeq e^*\Omega^1_{G/S}\oplus e^*\Omega^1_{G/S}
\]
induced by $m$, is given by
\[
\omega\mapsto(\omega,\omega),
\]
form which the lemma follows.
\end{proof}

We end this section by proving \autoref{theoremStabilizersOfJets}(\ref{computationOfStabliizersOfJets}).

\begin{proof}[Proof of \autoref{theoremStabilizersOfJets}(\ref{computationOfStabliizersOfJets})]
Set $r = \max(\ord_{\sI}(\varphi), \het^{(0)}_{L\sigma^*L_{\cI/\cX}}(\varphi))$. We have $n \geq \ord_\sI(\varphi)$ and $r \geq \ord_{\sI}(\varphi)$, so \autoref{autosEventuallyTruncateToSection} implies that for every $k'$-algebra $A$, the $A$-valued points of the stabilizer of $\theta_n(\varphi) \in \sL_n(\cX)(k')$, are given by deformations $\Spec(A[t]/(t^{n+1})) \to \cI_\cX$ (up to 2-isomorphism) over $\cX$ of the composition
\[
	\Spec(A[t]/(t^{n-r+1}) \to \Spec(k'[t]/(t^{n-r + 1})) \xrightarrow{\varphi_{n-r}} \cX \xrightarrow{\sigma} \cI_\cX,
\]
where $\varphi_{n-r}$ is the jet associated to $\theta_{n-r}(\varphi) \in \sL_{n-r}(\cX)(k')$. Because $n \geq  2r - 1$, the closed immersion 
\[
	\Spec(A[t]/(t^{n-r+1})) \hookrightarrow \Spec(A[t]/(t^{n+1}))
\]
is defined by a square-zero ideal. (Note that the stabilizer of $\theta_n(\varphi)$ is a scheme, so we do not need to worry about 2-automorphisms: this is consistent with the next deformation theory computation keeping in mind that $\cI_\cX \to \cX$ is representable.) Therefore the stabilizer of $\theta_n(\varphi)$ is isomorphic as a scheme to
\[
	\Ext^0(L\varphi_{n-r}^*L\sigma^* L_{\cI_\cX /\cX}, (t^{n+1})/(t^{n-r+1})).
\]
This isomorphism is compatible with the group structures by \autoref{LieAlgArgument}, so the stabilizer of $\theta_n(\varphi)$ is isomorphic as an algebraic group to
\[
	\bG_{a, k'}^{\dim_{k'} \Ext^0(L\varphi_{n-r}^*L\sigma^* L_{\cI_\cX /\cX}, (t^{n+1})/(t^{n-r+1}))}.
\]
Therefore it is now sufficient to prove that
\[
	\dim_{k'} \Ext^0(L\varphi_{n-r}^*L\sigma^* L_{\cI_\cX /\cX}, (t^{n+1})/(t^{n-r+1})) = \het^{(0)}_{L\sigma^*L_{\cI_\cX/\cX}}(\varphi).
\]
Because $n+1 \geq r$, \autoref{replacingExtWithH} implies
\[
	\dim_{k'} \Ext^0(L\varphi_{n-r}^*L\sigma^* L_{\cI_\cX /\cX}, (t^{n+1})/(t^{n-r+1})) = \het^{(0)}_{r-1, L\sigma^*L_{\cI_\cX / \cX}}(\varphi).
\]
Note that because $\cI_\cX \to \cX$ is representable,
\[
	\het^{(1)}_{L\sigma^* L_{\cI_\cX / \cX}}(\varphi) = 0.
\]
Thus because $r \geq \het^{(0)}_{L\sigma^* L_{\cI_\cX / \cX}}(\varphi)$, \autoref{theoremMainPropertiesOfHeightFunctions}(\ref{theoremRelatingArcHeightFunctionToJetHeightFunction}) implies that
\[
	\het^{(0)}_{r-1, L\sigma^*L_{\cI_\cX / \cX}}(\varphi) = \het^{(0)}_{L\sigma^* L_{\cI_\cX / \cX}}(\varphi),
\]
completing the proof.
\end{proof}

\section{A technical theorem}

The main result in this section is \autoref{technicalLemma}, which is a key technical lemma in the proof of \autoref{theoremChangeOfVariables}(\ref{theoremPartChangeOfVariables}).

\begin{theorem}\label{technicalLemma}
Let $\cX$ be a smooth irreducible finite type Artin stack over $k$, let $Y$ be an irreducible finite type scheme over $k$, let $\pi: \cX \to Y$ be a morphism, let $\cU$ be an open substack of $\cX$ such that $\cU \to \cX \to Y$ is an open immersion, and let $\cC \subset |\sL(\cX)| \setminus |\sL(\cX \setminus \cU)| \subset \sL(\cX)$ be a cylinder. Let $h,e' \in \Z_{\geq 0}$, let $\sJ_Y$ be the $(\dim Y)$th fitting ideal of $\Omega^1_{Y/k}$, and assume that $\het^{(0)}_{L_{\cX/Y}}$ and $\ord_{\sJ_Y} \circ \sL(\pi)$ are equal to $h$ and $e'$, respectively, on $\cC$. Let $\ell \in \Z_{\geq 0}$ be such that $\cC$ is the preimage along $\theta_\ell$ of a constructible subset of $\sL_\ell(\cX)$. Assume that the map $\overline{\cC}(k') \to \sL(Y)(k')$ is injective for all field extensions $k'$ of $k$.

Let $k'$ be a field extension of $k$, let $n \in \Z_{\geq 0}$, let $\varphi, \varphi' \in \cC(k')$, and assume that $\theta_n(\sL(\pi)(\varphi)) = \theta_n(\sL(\pi)(\varphi'))$. If $n \geq \max(\ell + h, e')$, then $\theta_{n-e}(\varphi) \cong \theta_{n-e}(\varphi')$ for any $e \in \Z_{\geq0}$ satisfying $n \geq e \geq h$.
\end{theorem}

The proof of \autoref{technicalLemma} will occupy the entirety of this section. We claim that it suffices to prove:

\begin{proposition}\label{technicalLemmaProp}
With notation and hypotheses as in the first paragraph of \autoref{technicalLemma}. Let $k'$ be a field extension of $k$, let $n \in \Z_{\geq 0}$, let $\varphi\in \cC(k')$, let $\alpha=\pi\circ\varphi$, and let $\fm=(t)\subset k'\llbracket t \rrbracket$. 

Suppose $n \geq \max(h, e')$. Then if $\Phi\in\Hom(\alpha^*\Omega^1_{Y/k},\fm^{n+1}/\fm^{2(n+1)})$ and $\phi\in\Hom(\alpha^*\Omega^1_{Y/k},\fm^{n+1}/\fm^{n+2})$ is the induced map, then $\phi$ lifts to an element of the group $\Ext^0(L\varphi^*L_{\cX/k},\fm^{n+1-e}/\fm^{n+2})$.
\end{proposition}

Let us explain why the above result implies \autoref{technicalLemma}. Let $\varphi'$ be as in \autoref{technicalLemma} and let $\alpha'=\pi\circ\varphi'$. Since $\theta_{2(n+1)-1}(\alpha)$ and $\theta_{2(n+1)-1}(\alpha')$ define the same $n$-jet of $Y$, deformation theory allows us to view $\theta_{2(n+1)-1}(\alpha')$ as an element $\Phi\in\Hom(\alpha^*\Omega^1_{Y/k},\fm^{n+1}/\fm^{2(n+1)})$. Let $\widetilde{\phi}\in \Ext^0(L\varphi^*L_{\cX/k},\fm^{n+1-e}/\fm^{n+2})$ be the lift of $\phi$ provided by \autoref{technicalLemmaProp}. Perturbing $\theta_{n+1}(\varphi')$ by the element $\widetilde{\phi}$ in the deformation space, we arrive at a new element $\varphi''_{n+1}\in\sL_{n+1}(\cX)(k')$, and since $\cX$ is smooth, infinitesimal lifting allows us to lift $\varphi''_{n+1}$ to an arc $\varphi''\in\sL(\cX)(k')$. Furthermore, by construction, $\varphi''$ satisfies $\theta_{n-e}(\varphi)\cong\theta_{n-e}(\varphi'')$ and $\theta_{n+1}(\sL(\pi)(\varphi')) = \theta_{n+1}(\alpha') = \theta_{n+1}(\sL(\pi)(\varphi''))$.

Proceeding by induction (note that $\sL(\cX)$ is the inverse limit of the $\sL_n(\cX)$), we arrive at an element $\varphi'''\in\sL(\cX)(k')$ such that $\theta_{n-e}(\varphi)\cong\theta_{n-e}(\varphi''')$ and $\sL(\pi)(\varphi') = \sL(\pi)(\varphi''')$. Since $n-e\geq\ell$, $\cC$ is the preimage along $\theta_\ell$ of a constructible subset of $\sL_\ell(\cX)$, and $\varphi\in\cC(k')$, we see $\varphi'''\in\cC(k')$. Then since $\overline{\cC}(k') \to \sL(Y)(k')$ is injective and $\sL(\pi)(\varphi') = \sL(\pi)(\varphi''')$, we see $\varphi'\cong\varphi'''$. Therefore, $\theta_{n-e}(\varphi) \cong \theta_{n-e}(\varphi''')\cong \theta_{n-e}(\varphi')$, proving \autoref{technicalLemma}.

\vspace{1em}

\autoref{technicalLemmaProp} will be proved after a series of lemmas. Throughout the rest of this section, we fix the notation $D=\Spec(k'\llbracket t \rrbracket)$.

\begin{lemma}
\label{keyLemmaForLooijenga-twisted}
The following hold:
\begin{itemize}
\item $L^0\varphi^*L_{\cX/k}$ is a vector bundle,
\item $L^i\varphi^*L_{\cX/Y}$ is torsion for any $i\in\Z$, and
\item $L^{-1}\varphi^*L_{\cX/Y}=(\alpha^*\Omega^1_{Y/k})_{\tors}$.
\end{itemize}
Moreover, the natural map $L\alpha^*L_{Y/k}\to L\varphi^*L_{\cX/k}$ induces an exact sequence
\begin{equation}\label{eqn:keyLemmaForLooijenga-twisted}
0\to (\alpha^*\Omega^1_{Y/k})_{\tors}\to \alpha^*\Omega^1_{Y/k}\to L^0\varphi^*L_{\cX/k}\to L^0\varphi^*L_{\cX/Y}\to 0.
\end{equation}
\end{lemma}
\begin{proof}
Applying $L\varphi^*$ to the exact triangle
\[
L\pi^*L_{Y/k}\to L_{\cX/k}\to L_{\cX/Y},
\]
we have an exact triangle
\[
L\alpha^*L_{Y/k}\to L\varphi^*L_{\cX/k}\to L\varphi^*L_{\cX/Y}.
\]
Taking cohomology sheaves, %and using that $L^{-1}\varphi^*L_{\cX/k}=0$, 
we have the exact sequence 
\[
0\to L^{-1}\varphi^*L_{\cX/Y}\to \alpha^*\Omega^1_{Y/k}\stackrel{\sigma}{\to} L^0\varphi^*L_{\cX/k}\to L^0\varphi^*L_{\cX/Y}\to 0.
\]
To obtain (\ref{eqn:keyLemmaForLooijenga-twisted}), we must show $L^{-1}\varphi^*L_{\cX/Y}=(\alpha^*\Omega^1_{Y/k})_{\tors}$. This will follow upon showing $L^0\varphi^*L_{\cX/k}$ is torsion-free and $L^{-1}\varphi^*L_{\cX/Y}$ is torsion. Indeed, then
\[
(\alpha^*\Omega^1_{Y/k})_{\tors}\subseteq\ker(\sigma)=L^{-1}\varphi^*L_{\cX/Y}\subseteq(\alpha^*\Omega^1_{Y/k})_{\tors}.
\]

Next, consider the cartesian diagram
\[
\xymatrix{
V\ar[r]^-{\widetilde{\varphi}}\ar[d]_-{q} & \widetilde{X}\ar[d]^-{p}\\
D\ar[r]^-{\varphi} & \cX
}
\]
where $p$ is a smooth cover. % and $V$ is a scheme since $\varphi$ is representable. 
Then we have an exact triangle
\[
p^*L_{\cX/k}\to\Omega^1_{\widetilde{X}/k}\to\Omega^1_{\widetilde{X}/\cX}
\]
and applying $L\widetilde{\varphi}^*$, we have an exact triangle
\[
q^*L\varphi^*L_{\cX/k}\to\widetilde{\varphi}^*\Omega^1_{\widetilde{X}/k}\to\widetilde{\varphi}^*\Omega^1_{\widetilde{X}/\cX}.
\]
As a result, $q^*L^0\varphi^*L_{\cX/k}\to\widetilde{\varphi}^*\Omega^1_{\widetilde{X}}$ is injective and since $\widetilde{\varphi}^*\Omega^1_{\widetilde{X}}$ is a vector bundle, $q^*L^0\varphi^*L_{\cX/k}$ is torsion-free. Since $q$ is smooth, this implies $L^0\varphi^*L_{\cX/k}$ is a torsion-free sheaf on $D$, hence a vector bundle.

It remains to prove that $L^m\varphi^*L_{\cX/Y}$ is torsion for all $m\in\Z$. To do so, it is enough to show $L^m\varphi^*L_{\cX/Y}$ is supported on the closed fiber of $D$. In other words, we need only show $\eta^*L^m\varphi^*L_{\cX/Y}=0$, where $\eta\colon D^\circ\to D$ is the generic point. Since the generic point of $\varphi$ factors through $\cU$, we have a commutative diagram
\[
\xymatrix{
D^\circ\ar[d]_-{\eta}\ar[r]^-{\varphi^\circ} & \cU\ar[d]^-{j}\\
D\ar[r]^-{\varphi} & \cX
}
\]
So, $\eta^*L^m\varphi^*L_{\cX/Y}=L^m(\varphi^\circ)^*j^*L_{\cX/Y}$. From the exact triangle
\[
j^*L_{\cX/Y}\to L_{\cU/\cX}\to L_{\cU/Y}
\]
and the fact that $\cU\to\cX$ and $\cU\to Y$ are open immersions, we have $ L_{\cU/\cX}=L_{\cU/Y}=0$, and hence $j^*L_{\cX/Y}=0$.
\end{proof}

\begin{lemma}
\label{extendingMapZellGraded}
Let $M$ be a finitely generated $k'\llbracket t \rrbracket$-module, $F$ be a finite-rank free $k'\llbracket t \rrbracket$-module, and 
\[
0\to M_{\tors}\to M\to F\to Q\to0
\]
be an exact sequence. Let $h=\len(Q)$, $e'=\len(M_{\tors})$, $e\geq h$, and $n\geq\max(e,e')$. If $\Phi\in\Hom(M,\fm^{n+1}/\fm^{2(n+1)})$ and $\phi\in\Hom(M,\fm^{n+1}/\fm^{n+2})$ is the induced map, then $\phi$ lifts to an element of $\Hom(F,\fm^{n+1-e}/\fm^{n+2})$.
\end{lemma}
\begin{proof}
To begin, we have $\Phi(M_{\tors})\subseteq\fm^{2(n+1)-e'}/\fm^{2(n+1)}$ and since $n\geq e'$, it follows that $\phi$ kills $M_{\tors}$. As a result, $\phi$ induces an element $\overline{\phi}$ of the group $\Hom(M/M_{\tors},\fm^{n+1}/\fm^{n+2})$. Replacing $M$ with $M/M_{\tors}$, we may assume $M$ is free, and so we have a short exact sequence
\[
0\to M\to F\to Q\to 0.
\]

Let $v_1,\dots,v_d$ be a basis for $F$. By the structure theorem for modules over a PID, we have $a_1,\dots,a_d\in k^*$ and non-negative integers $h_1,\dots,h_d$ such that $a_1t^{h_1}v_1,\dots,a_dt^{h_d}v_d$ is a basis for $M$. Note that $h:=\len(Q)=\sum_{i=1} h_i$. We have $\phi(a_it^{h_i})=b_it^{n+1}$ for $b_i\in k$. Thus, we may extend $\phi$ to a map
\[
\widetilde{\phi}\colon E\to\fm^{n+1-e}/\fm^{n+2}
\]
by letting
\[
\widetilde{\phi}(v_i)=\frac{b_i}{a_i}t^{n+1-h_i}.\qedhere
\]
\end{proof}

We prove \autoref{technicalLemmaProp}, hence completing the proof of \autoref{technicalLemma}.

\begin{proof}[{Proof of \autoref{technicalLemmaProp}}]
Let $\cM$ be any coherent sheaf on $D$. We have a natural map $L\alpha^*L_{Y/k}\to L\varphi^*L_{\cX/k}$ which induces a map
\begin{equation}\label{eqn:liftFromAlphaToStack1}
\Ext^0(L\varphi^*L_{\cX/k},\cM)\to\Ext^0(L\alpha^*L_{Y/k},\cM)=\Hom(\alpha^*\Omega^1_{Y/k},\cM)
\end{equation}

We first claim that $L\alpha^*L_{Y/k}\to L\varphi^*L_{\cX/k}$ factors through $L^0\varphi^*L_{\cX/k}$, and hence the map in (\ref{eqn:liftFromAlphaToStack1}) factors as
\begin{equation}\label{eqn:liftFromAlphaToStack2}
\Ext^0(L\varphi^*L_{\cX/k},\cM)\stackrel{F}{\to}\Ext^0(L^0\varphi^*L_{\cX/k},\cM)\stackrel{G}{\to}\Hom(\alpha^*\Omega^1_{Y/k},\cM).
\end{equation}
Indeed, we have a morphism of exact triangles
\[
\xymatrix{
\tau_{\leq0}L\alpha^*L_{Y/k}\ar[r]\ar[d] & L\alpha^*L_{Y/k}\ar[r]\ar[d] & \tau_{>0}L\alpha^*L_{Y/k}\ar[d]\\
\tau_{\leq0}L\varphi^*L_{\cX/k}\ar[r] & L\varphi^*L_{\cX/k}\ar[r] & \tau_{>0}L\varphi^*L_{\cX/k}
}
\]
where $\tau$ denotes the canonical truncation. Since $L\alpha^*L_{Y/k}$ is concentrated in degrees at most $0$, the map $\tau_{\leq0}L\alpha^*L_{Y/k}\to L\alpha^*L_{Y/k}$ is the identity map; since $L\varphi^*L_{\cX/k}$ is concentrated in degrees at least $0$, $\tau_{\leq0}L\varphi^*L_{\cX/k}=L^0\varphi^*L_{\cX/k}$.

Next, we show that the map $F$ in (\ref{eqn:liftFromAlphaToStack2}) is surjective. To see this, we apply the spectral sequence from \cite[Tag 07A9]{stacks-project}
\[
E_2^{pq}=\Ext^p(L^{-q}\varphi^*L_{\cX/k},\cM)\Rightarrow \Ext^{p+q}(L\varphi^*L_{\cX/k},\cM),
\]
which yields an exact sequence
\begin{align*}
0\to&\Ext^1(L^1\varphi^*L_{\cX/k},\cM)\to \Ext^0(L\varphi^*L_{\cX/k},\cM)\\
&\stackrel{F}{\to} \Hom(L^0\varphi^*L_{\cX/k},\cM)\to \Ext^2(L^1\varphi^*L_{\cX/k},\cM).
\end{align*}
Since $k'\llbracket t \rrbracket$ is a PID, $\Ext^2(L^1\varphi^*L_{\cX/k},\cM)$ vanishes and so $F$ is surjective.

Lastly, note that $h=\len(L^0\varphi^*L_{\cX/k})$ and $e'=\len((\alpha^*\Omega^1_{Y/k})_{\tors})$. So, by \autoref{keyLemmaForLooijenga-twisted} and \autoref{extendingMapZellGraded}, we see $\phi$ extends to an element $\phi'$ of the group $\Hom(L^0\varphi^*L_{\cX/k},\fm^{n+1-e}/\fm^{n+2})$. Since $F$ is surjective, $\phi'$ lifts to an element of $\Ext^0(L\varphi^*L_{\cX/k},\fm^{n+1-e}/\fm^{n+2})$.
\end{proof}

\section{Motivic change of variables formula}

In this section we complete the proof of \autoref{theoremChangeOfVariables}(\ref{theoremPartChangeOfVariables}). We give the proof in \autoref{subsec:proofOfPartChangeVars} after proving a preliminary result in \autoref{subsec:stratsByGerbes}.

\subsection{Stratifications by trivial gerbes}\label{subsec:stratsByGerbes}

The goal of this subsection is to prove the following proposition. We expect this result is well-known, but we could not find it stated in the literature.%, so we include an outline of the proof.

\begin{proposition}\label{stratifyLemmaIntoGaGerbes}
Let $\cY$ be a finite type Artin stack over $k$ such that the map $\cI_\cY \to \cY$ is separated, let $r \in \Z_{\geq 0}$, and suppose that for every field extension $k'$ of $k$ and every $y \in \cY(k')$, the stabilizer of $y$ is isomorphic to $\bG_{a,k'}^r$. Then $\cY$ can be stratified into locally closed substacks $\cY_1, \dots, \cY_\ell$ such that each $\cY_i$ is isomorphic to $Y_i \times_k B\bG_{a,k}^r$ for some finite type $k$-scheme $Y_i$.
\end{proposition}

We begin with a couple of preliminary lemmas.

\begin{lemma}
\label{equalityMapsOnFibers}
Let $S$ be a scheme, and let $X$ and $Y$ be $S$-schemes with $X$ reduced and $Y\to S$ separated. If $f,g\colon X\to Y$ are two $S$-morphisms and for all fibers $s\colon\Spec k(s)\to S$, we have equality of the base change maps $f_s=g_s$, then $f=g$.
\end{lemma}
\begin{proof}
Consider the cartesian diagram
\[
\xymatrix{
Z\ar[r]\ar[d]_-{i} & Y\ar[d]^-{\Delta_{Y/S}}\\
X\ar[r]^-{(f,g)} & Y\times_S Y
}
\]
Note that $f=g$ if and only if $i$ has a section. Since $Y$ is separated over $S$, we see $i$ is a closed immersion. As $X$ is reduced, $i$ has a section if and only if $i$ is surjective. Since surjectivity of $i$ can be checked on fibers, it follows that $f=g$.
\end{proof}

\begin{lemma}\label{groupCommutativeIfFibersCommutative}
Let $k'$ be a characteristic $0$ field and let $S$ be a reduced $k'$-scheme which is locally of finite presentation. Suppose $f\colon G\to S$ is a separated group scheme which is flat and locally of finite presentation. If every fiber $f_s\colon G_s\to\Spec k'(s)$ is a commutative group scheme, then $G$ is a commutative.
\end{lemma}
\begin{proof}
Let $m\colon G\times_S G\to G$ be the multiplication map and $\tau\colon G\times_S G\to G\times_S G$ be the map $\tau(g,h)=(h,g)$. To prove $G$ is commutative, we must show $m\circ\tau=m$. By Cartier's Theorem (see e.g., \cite[Tag 047N]{stacks-project}), every fiber $f_s\colon G_s\to\Spec k'(s)$ is a smooth group scheme. Since $f$ flat and locally of finite presentation, \cite[Tag 01V8]{stacks-project} shows that $f$ is smooth; in particular, $G\times_S G\to S$ is smooth. Since $S$ is reduced, $G\times_S G$ is as well. Lastly, the maps $m,m\circ\tau\colon G\times_S G\to G$ are equal on fibers, so $m\circ\tau=m$ by \autoref{equalityMapsOnFibers}.
\end{proof}

\begin{proof}[{Proof of \autoref{stratifyLemmaIntoGaGerbes}}]
First, replacing $\cY$ by an irreducible component, we may assume $\cY$ is irreducible. We may also replace $\cY$ by $\cY_\red$, so we will assume $\cY$ is reduced. By Noetherian induction, it suffices to show there exists a non-empty open substack of $\cY$ of the form $Y \times_k B\bG_{a,k}^r$ for some finite type $k$-scheme $Y$. Since $\cI_\cY\to\cY$ is quasi-compact, replacing $\cY$ by an open substack, we may assume it is a gerbe by \cite[Tag 06RC]{stacks-project}. Let $p\colon\cY\to W$ be the gerbe structure map, so $W$ is an algebraic space. Because $\cY$ is reduced and gerbe structure maps are stable under base change, we may assume $W$ is reduced.
%Letting $p\colon\cY\to W$ be the morphism from \cite[Tag 06QD]{stacks-project}, we have that $W$ is an algebraic space and $p$ is both a gerbe and a coarse space map. %Remark 6.13 of \url{https://arxiv.org/pdf/0708.3333.pdf}
Since $\cY$ is irreducible and finite type, so is $W$. Since $W$ is quasi-separated, it has a dense open scheme by \cite[Tag 06NH]{stacks-project}. We therefore assume $W$ is a scheme.

Let $\Spec(K)\to W$ be the generic point. We claim that $\cY_K:=\cY\times_W\Spec(K)$ is isomorphic to $B\bG_{a,K}^r$. Assuming this claim for the moment, we have a morphism $\alpha\colon\cY_K\to B\bG_{a,K}^r$. By \cite[Proposition B.2]{Rydh2015}, there exists an open subscheme $U\subset W$ for which $\alpha$ extends to a map $\cY\times_W U\to B\bG_{a,U}^r$. Since $\cY\to W$ is of finite presentation and $\alpha$ is an isomorphism, by \cite[Proposition B.3]{Rydh2015}, after possibly shrinking $U$, we may assume $\cY\times_Y U\to B\bG_{a,U}^r$ is an isomorphism, thereby producing our desired open substack of $\cY$ isomorphic to $B\bG_{a,U}^r=U\times_k B\bG_{a,k}^r$.

It remains to prove our claim, i.e.~we have reduced to showing that if $W=\Spec(K)$, then $\cY$ is the trivial $\bG_{a,K}^r$-gerbe. Note that because $k$ and thus $K$ have characteristic 0, the gerbe structure map $\cY \to W$ is smooth, so in this case where $W = \Spec(K)$, the stack $\cY$ is reduced. Let $\rho\colon\widetilde{Y}\to\cY$ be a smooth cover. We first show that all $S$-valued points of $\cY$ have commutative automorphism groups. This may be checked smooth locally on $S$, and hence it is enough to check that commutativity of the automorphism group $G$ of the $\widetilde{Y}$-valued point corresponding to $\rho$. We see $G=I_\cY\times_\cY\widetilde{Y}$. Since the stabilizer groups of all field-valued points are commutative, $G\to\widetilde{Y}$ is a commutative group scheme by \autoref{groupCommutativeIfFibersCommutative}.

Therefore, \cite[Tag 0CJY]{stacks-project} tells us that $\cY$ is a $\cG$-gerbe for some sheaf of abelian groups $\cG$ on $\Spec(K)$. Let $K'/K$ be an \'etale extension that trivializes the gerbe $\cY$. Since the $K'$-point of $\cY$ yielding the trivialization has stabilizer group scheme given by $\bG_{a,K'}^r$, we see $\cG$ is a twisted form of $\bG_{a,K}^r$. Such twisted forms are classified by $H^1_{et}(K,\Aut(\bG_a^r))=H^1_{et}(K,\GL_r)$, which vanishes by Hilbert's Theorem 90. Thus, $\cG=\bG_a^r$. Lastly, $\bG_a^r$-gerbes on $\Spec(K)$ are classified by $H^2(\Spec(K),\cO^{\oplus r})=0$. It follows that $\cY$ is the trivial $\bG_a^r$-gerbe.
\end{proof}

\subsection{Proof of \autoref{theoremChangeOfVariables}(\ref{theoremPartChangeOfVariables})}\label{subsec:proofOfPartChangeVars}

For the remainder of this paper, let $\pi: \cX \to Y$, $\cU$, $\cC$, $D$, and $\sigma$ be as in the statement of \autoref{theoremChangeOfVariables}, let $\sJ_Y$ be the $(\dim Y)$th fitting ideal of $\Omega^1_{Y/k}$, let $\sJ$ denote the ideal sheaf on $\cI_\cX$ defined by the closed substack $\cX \xrightarrow{\sigma} \cI_\cX$, and let $\sI$ be a quasi-coherent ideal sheaf on $\cX$ such that $\sI \cO_{\cI_\cX} \cdot \sJ = 0$ and the closed substack of $\cX$ defined by $\sI$ has underlying topological space equal to $|\cX| \setminus |\cU|$. Note that such an ideal $\sI$ is guaranteed to exist by \autoref{theoremStabilizersOfJets}(\ref{existenceOfIdealControllingStabilizers}).

We begin with the following key case of \autoref{theoremChangeOfVariables}(\ref{theoremPartChangeOfVariables}), where the functions of interest are all constant on $\cC$.

\begin{proposition}\label{propositionchangeOfVariablesConstantCase}
Suppose that $\ord_{\sJ_Y}$ is constant on $D$ and that $\het_{L_{\cX/Y}}^{(0)} + \het_{L_{\cX/Y}}^{(1)}$, $\het_{L_{\cX/Y}}^{(1)}$, $\het_{L\sigma^*L_{\cI_\cX/\cX}}^{(0)}$, and $\ord_\sI$ are constant on $\cC$. Then
\[
	\mu_Y(D) = \int_{\cC}\bL^{\het^{(0)}_{L\sigma^*L_{\cI_\cX/\cX}} - \het^{(0)}_{L_{\cX/Y}}} \diff \mu_\cX.
\]
\end{proposition}

\begin{proof}
Let $e' \in \Z_{\geq 0}$ be such that $\ord_{\sJ_Y}$ is equal to $e'$ on $D$, and let $h_{01}, h_{1}, h_0, r \in \Z_{\geq 0}$ be such that $\het_{L_{\cX/Y}}^{(0)} + \het_{L_{\cX/Y}}^{(1)}$, $\het_{L_{\cX/Y}}^{(1)}$, $\het_{L\sigma^*L_{\cI_\cX/\cX}}^{(0)}$, and $\ord_{\sI}$ are equal to $h_{01}$, $h_1$, $h_0$ and $r$, respectively, on $\cC$. Because $\cC$ is a cylinder, there exists some $\ell \in \Z_{\geq 0}$ such that $\theta_\ell^{-1}(\theta_\ell(\cC)) = \cC$. Set $h = h_{01} - h_1$, and note that $\het^{(0)}_{L_{\cX/Y}}$ is equal to $h$ on $\cC$, so in particular $h \in \Z_{\geq 0}$. Set $e = h_{01} + h_1$. Because the (isomorphic) image of $\cU$ in $Y$ is smooth, we have that $D \cap \sL(Y \setminus Y_{\mathrm{sm}}) = \emptyset$, where $Y_\mathrm{sm}$ is the smooth locus of $Y$. Thus there exists some $n_D \in \Z_{\geq 0}$ that satisfies the following: for any field extension $k'$ of $k$, any $n \geq n_D$, and any $\psi_n \in (\theta_n(D))(k')$, there exists some $\psi \in D(k')$ such that $\theta_n(\psi) = \psi_n$, see e.g. \cite[Lemma 7.5]{SatrianoUsatine}.

Fix any $n \geq \max(\ell+e, r+e, 2\max(r, h_0)-1+e, e', n_D, 2e-1)$. Because $\dim \cX = \dim Y$, it is sufficient to show that
\[
	\e(\theta_n(D)) = \bL^{h_0 - (h_{01} - h_1)}\e(\theta_n(\cC)).
\]
Set $\cC_{n-e} = \theta_{n-e}(\cC) \subset |\sL_n(\cX)|$. For every field extension $k'$ of $k$, every $k'$ valued point of $\cC_{n-e}$ lifts to a $k'$-valued point of $\cC$ because $\cX$ is smooth and $n-e \geq \ell$. Thus because $n-e \geq \max(r, 2\max(r, h_0)-1)$, \autoref{theoremStabilizersOfJets}(\ref{computationOfStabliizersOfJets}) implies that every $k'$-valued point of $\cC_{n-e}$ has stabilizer isomorphic to $\bG_{a,k}^{h_0}$. Because $\cX$ has separated diagonal, $\sL_{n-e}(\cX)$ has separated diagonal \cite[Proposition 3.8(xx)]{Rydh}. Thus \autoref{stratifyLemmaIntoGaGerbes} implies that $\cC_{n-e}$ can be stratified into (finitely many) locally closed substacks $\{\cV_i\}_i$ of $\sL_{n-e}(\cX)$ such that each $\cV_i$ is isomorphic to $V_i \times_k B\bG_{a,k}^{h_0}$ for some finite type $k$-scheme $V_i$. Let each $V_i \to \cV_i$ be the quotient map by the trivial $\bG_{a,k}^{h_0}$-action on $V_i$, and let each $\cV_i \hookrightarrow \sL_{n-e}(\cX)$ be the inclusion. Also for each $i$, let $\cZ_i \to \cW_i \hookrightarrow \sL_n(\cX)$ be the base change of $V_i \to \cV_i \hookrightarrow \sL_{n-e}(\cX)$ along $\theta^n_{n-e}: \sL_n(\cX) \to \sL_{n-e}(\cX)$. Because $\cZ_i \to \cW_i$ is a $\bG_{a,k}^{h_0}$-torsor and $\bG_{a,k}^{h_0}$ is a special group,
\[
	\e(\cZ_i) = \bL^{h_0}\e(\cW_i).
\]
Because $n - e \geq \ell$ and $\cX$ is smooth, $\theta_n(\cC) = (\theta^{n}_{n-e})^{-1}(\cC_{n-e})$, so the $\cW_i$ form a partition of $\theta_n(\cC)$. Thus
\begin{equation}\label{equationOneInChangeOfVariablesConstantCase}
	\sum_i \e(\cZ_i) = \bL^{h_0} \e(\theta_n(\cC)).
\end{equation}
Now set $Q_i = \sL_n(\pi)(|\cW_i|) \subset |\sL_n(Y)|$. By Chevalley's theorem, the $Q_i$ are constructible subsets of $|\sL_n(Y)|$. We will now show that the $Q_i$ partition $\theta_n(D)$. The $Q_i$ cover $\theta_n(D)$ because $\cC$ surjects onto $D$ and because the $|\cW_i|$ cover $\theta_n(\cC)$. We just need to show that the $Q_i$ are mutually disjoint. To do this, it is sufficient to show that if $k'$ is a field extension of $k$ and $\varphi_n, \varphi'_n \in \theta_n(\cC)(k')$ have the same image in $\sL_n(Y)$, then $\theta^n_{n-e}(\varphi_n) \cong \theta^n_{n-e}(\varphi'_n)$. Because $\cX$ is smooth and $n \geq \ell$, there exist $\varphi, \varphi' \in \cC(k')$ that truncate to $\varphi_n, \varphi'_n$, respectively. Because $n \geq \max(\ell + e, e') \geq \max(\ell + h, e')$ and $n \geq e \geq h$, \autoref{technicalLemma} implies that $\theta_{n-e}(\varphi) \cong \theta_{n-e}(\varphi')$, so we have shown that the $Q_i$ partion $\theta_n(D)$. Therefore
\begin{equation}\label{equationTwoInChangeOfVariablesConstantCase}
	\sum_i \e(Q_i) = \e(\theta_n(D)).
\end{equation}
Our goal is to relate $\e(Q_i)$ to $\e(\cZ_i)$. Because the composition $\cZ_i \to \cW_i \hookrightarrow \sL_n(\cX) \xrightarrow{\sL_n(\pi)} \sL_n(Y)$ maps $|\cZ_i|$ onto $Q_i$, we will accomplish this by controlling the fiber (with reduced structure) of $\cZ_i \to \sL_n(Y)$ over every point of $Q_i$.

Fix some $i$, some field extension $k'$ of $k$, and some $\psi_n \in Q_i(k')$. Because $n \geq n_D$, there exists some $\psi \in D(k')$ that truncates to $\psi_n$. By the assumption on $\overline{\cC}(k') \to D(k')$, there is some $\varphi \in \cC(k')$ whose image is $\psi$. Set $\varphi_n = \theta_n(\varphi) \in \theta_n(\cC)(k')$ and $\varphi_{n-e} = \theta_{n-e}(\varphi)$. We have already shown that the $|\cW_j|$ have disjoint images in $\sL_n(Y)$, so the fact that $\varphi_n$ has image $\psi_n \in Q_i(k')$ implies that $\varphi_n \in \cW_i(k')$. Because $\cZ_i \to \cW_i$, as a base change of $V_i \to \cV_i$, is representable and a $\bG_{a,k}^{h_0}$-torsor and $\bG_{a,k}^{h_0}$ is a special group, there exists some $z \in \cZ_i(k')$ whose image in $\cW_i$ is 2-isomorphic to $\varphi_n$. Set $v \in V_i(k')$ to be the image of $z$ along $\cW_i \to V_i$. Now consider the following 2-commutative diagram.
\begin{center}
\begin{tikzcd}
& \cZ_i \arrow[r] \arrow[d] & \cW_i \arrow[r] \arrow[d] & \sL_n(\cX) \arrow[r, "\sL_n(\pi)"] \arrow[d, "\theta_n"] & \sL_n(Y) \arrow[d, "\theta_n"]\\
\Spec(k') \arrow[ur, "z"] \arrow[r, "v"] \arrow[rrr, bend right, "\varphi_{n-e}"] \arrow[urrr, bend left = 60, "\varphi_n"] & V_i \arrow[r] & \cV_i \arrow[r] & \sL_{n-e}(\cX) \arrow[r, "\sL_{n-e}(\pi)"] & \sL_{n-e}(Y)
\end{tikzcd}
\end{center}
Let $\cF$ be the fiber of $\cZ_i \to \sL_n(Y)$ over $\psi_n$, and let $\cH$ be the fiber of $\cZ_i \to V_i$ over $v$. Because $\sL_n(Y)$ and $V_i$ are schemes, $\cF$ and $\cH$ are closed substacks of $\cZ_i \otimes_k k'$. We will show that $|\cF| \subset |\cH|$.

Suppose for some field extension $k''$, we have $z' \in \cF(k'')$, so the image of $z'$ in $\sL_n(Y)$ is $\psi_n \otimes_{k'} k''$. We already saw (when proving that the $Q_j$ were disjoint) that any two $k''$-points of $\theta_n(\cC)$ that have the same image in $\sL_n(Y)$ must have 2-isomorphic $(n-e)$-truncations (this was where we used \autoref{technicalLemma} above). Thus $z'$ and $z \otimes_{k'} k''$ have 2-isomorphic images in $\sL_{n-e}(\cX)(k'')$. Because the quotient map $V_i \to \cV_i$ is a section of $\cV_i = V_i \times_k B\bG_{a,k}^{h_0} \to V_i$, it induces an injection $V_i(k'') \to \overline{\cV_i}(k'')$. Also $\cV_i \hookrightarrow \sL_{n-e}(\cX)$ induces an injection $\overline{\cV_i}(k'') \to \overline{\sL_{n-e}(\cX)}(k'')$. Therefore $z'$ and $z \otimes_{k'} k''$ have the same image in $V_i$, namely $v \otimes_{k'} k''$. Thus $z' \in \cH(k'')$, and we have shown that $|\cF| \subset |\cH|$, so $\cF_\red$ is a closed substack of $\cH$.

Recall that we want to compute $\cF_\red$. We will now accomplish this by computing $\cH$ and understanding $|\cF|$ as a subset of $|\cH|$. Because $n \geq 2e - 1$, we have that $\Spec(k'[t]/(t^{n-e+1})) \hookrightarrow \Spec(k'[t]/(t^{n+1})$ is defined by a square-zero ideal. Thus $\cH$, which is isomorphic to the fiber of $\sL_n(\cX) \to \sL_{n-e}(\cX)$ over $\theta_{n-e}(\varphi_n)$, can be identified with the deformation space (see e.g., the proof of \cite[Lemma 3.34]{SatrianoUsatine})
\[
	\Ext^0(L\varphi_{n-e}^* L_{\cX/k}, I ) \times_{k'} B\Ext^{-1}(L\varphi_{n-e}^* L_{\cX/k}, I ),
\]
where $I = (t^{n-e+1})/ (t^{n+1})$ and by slight abuse of notation, we are using $L\varphi_{n-e}^*$ to mean the derived pullback along the associated jet $\Spec(k'[t]/(t^{n-e+1})) \to \cX$. Similarly the fiber $H$ of $\sL_n(Y) \to \sL_{n-e}(Y)$ over $\theta^n_{n-e}(\psi_n)$ can be identified with the deformation space
\[
	\Ext^0(L\varphi_{n-e}^* L\pi^* L_{Y/k}, I )
\]
such that the map $\cH \to H$ is identified with the projection $\Ext^0(L\varphi_{n-e}^* L_{\cX/k}, I ) \times_{k'} B\Ext^{-1}(L\varphi_{n-e}^* L_{\cX/k}, I ) \to \Ext^0(L\varphi_{n-e}^* L_{\cX/k}, I )$ followed by the linear map 
\[
	f: \Ext^0(L\varphi_{n-e}^* L_{\cX/k}, I ) \to \Ext^0(L\varphi_{n-e}^* L\pi^* L_{Y/k}, I )
\]
induced by $L\pi^* L_{Y/k} \to L_{\cX/k}$, and $\cF_\red$ is identified with the closed substack of $\cH$
\[
	\ker f  \times_{k'} B\Ext^{-1}(L\varphi_{n-e}^* L_{\cX/k}, I ).
\]
Therefore
\[
	\cF_\red \cong \bA_{k'}^{\dim_{k'} \ker f} \times_{k'} B\bG_{a,k'}^{\dim_{k'} \Ext^{-1}(L\varphi_{n-e}^* L_{\cX/k}, I )}. 
\]
By the exact triangle $L\pi^*L_{Y/k} \to L_{\cX / k} \to L_{\cX / Y}$ and the fact that $Y$ is a scheme, we have that
\[
	\ker f \cong \Ext^0(L\varphi_{n-e}^* L_{\cX / Y}, I)
\]
and
\[
	\Ext^{-1}(L\varphi_{n-e}^* L_{\cX/k}, I ) \cong \Ext^{-1}(L\varphi_{n-e}^* L_{\cX/Y}, I ).
\]
Thus \autoref{replacingExtWithH} implies that
\[
	\dim_{k'} \ker f = \het^{(0)}_{e-1, L_{\cX/Y}}(\theta_{e-1}(\varphi)),
\]
and
\[
	\dim_{k'}\Ext^{-1}(L\varphi_{n-e}^* L_{\cX/k}, I ) = \het^{(1)}_{e-1, L_{\cX/Y}}(\theta_{e-1}(\varphi)),
\]
so
\[
	\cF_\red \cong \bA_{k'}^{\het^{(0)}_{e-1, L_{\cX/Y}}(\theta_{e-1}(\varphi))} \times_{k'} B\bG_{a,k'}^{\het^{(1)}_{e-1, L_{\cX/Y}}(\theta_{e-1}(\varphi))}. 
\]
Because $e \geq h_{01}+h_1$ and $\het^{(2)}_{L_{\cX/Y}}(\varphi) = 0$, \autoref{theoremMainPropertiesOfHeightFunctions}(\ref{theoremRelatingArcHeightFunctionToJetHeightFunction}) implies
\[
	\het^{(0)}_{e-1, L_{\cX/Y}}(\theta_{e-1}(\varphi)) = h_{01},
\]
and
\[
	\het^{(1)}_{e-1, L_{\cX/Y}}(\theta_{e-1}(\varphi)) = h_1,
\]
so
\[
	\cF_\red \cong \bA_{k'}^{h_{01}} \times_{k'} B\bG_{a,k'}^{h_1}. 
\]
Now that we have computed $\cF_\red$ and shown that it does not depend on $\psi_n$, we can finally relate $\e(\cZ_i)$ and $\e(Q_i)$. Specifically \cite[Remark 2.7 and Proposition 2.8]{SatrianoUsatine} implies
\[
	\e(\cZ_i) = \e(\bA_{k}^{h_{01}} \times_k B\bG_{a,k}^{h_1}) \e(Q_i) = \bL^{h_{01} - h_1} \e(Q_i)
\]
where the second equality is because $\bG_{a,k}^{h_1}$ is a special group. This along with equations (\ref{equationOneInChangeOfVariablesConstantCase}) and (\ref{equationTwoInChangeOfVariablesConstantCase}) give
\[
	\e(\theta_n(D)) = \bL^{h_0 - (h_{01} - h_1)}\e(\theta_n(\cC)),
\]
completing the proof of this proposition.
\end{proof}

We can now finish the proof of \autoref{theoremChangeOfVariables}(\ref{theoremPartChangeOfVariables}) by reducing to \autoref{propositionchangeOfVariablesConstantCase}. The remainder of the argument is essentially standard, but we still briefly outline it for completeness.

\begin{proof}[Proof of \autoref{theoremChangeOfVariables}(\ref{theoremPartChangeOfVariables})]
Because the (isomorphic) image of $\cU$ in $Y$ is smooth, we have that $D \cap \sL(Y \setminus Y_{\mathrm{sm}}) = \emptyset$, where $Y_\mathrm{sm}$ is the smooth locus of $Y$, so $\ord_{\sJ_Y}: D \to \Z$ is constructible. Thus $D$ can be partitioned into finitely many cylinders where $\ord_{\sJ_Y}$ is constant, and because the preimage of a cylinder along $\sL(\pi)$ is a cylinder, it is straightforward to reduce to the case where $\ord_{\sJ_Y}$ is constant on $D$. Thus we assume $\ord_{\sJ_Y}$ is constant on $D$.

Recall that $\het_{L_{\cX/Y}}^{(0)} + \het_{L_{\cX/Y}}^{(1)}$, $\het_{L_{\cX/Y}}^{(1)}$, and $\het_{L\sigma^*L_{\cI_\cX/\cX}}^{(0)}$ are constructible functions on $\cC$ (see the proof of \autoref{theoremChangeOfVariables}(\ref{theoremMeasurablePartOfCOV})). Also $\ord_{\sI}: \cC \to \Z$ is constructible by \autoref{remarkOrdIsConstructible}. Thus there exists a finite partition of $\cC$ into cylinders $\{\cC_i\}_i$ such that $\het_{L_{\cX/Y}}^{(0)} + \het_{L_{\cX/Y}}^{(1)}$, $\het_{L_{\cX/Y}}^{(1)}$, $\het_{L\sigma^*L_{\cI_\cX/\cX}}^{(0)}$, and $\ord_{\sI}$ are all constant on each $\cC_i$.

A straightforward argument (essentially identical to a standard argument in the scheme case) using  \autoref{technicalLemma} (see e.g., the use of \autoref{technicalLemma} in the proof of \autoref{propositionchangeOfVariablesConstantCase}) and the fact that there are only finitely many $\cC_i$ shows that there exists some $n \in \Z_{\geq 0}$ such that the $\sL_n(\pi)(\theta_n(\cC_i))$ are pairwise disjoint. Therefore $D$ is partitioned into the cylinders $D_i = D \cap \theta_n^{-1}(\sL_n(\pi)(\theta_n(\cC_i)))$, where we note that each $\sL_n(\pi)(\theta_n(\cC_i))$ is constructible by Chevalley's theorem. Then $\cC_i \to D_i$ satisfies the hypotheses of \autoref{propositionchangeOfVariablesConstantCase}, and the theorem follows by summing over $i$.
\end{proof}

\bibliographystyle{alpha}
\bibliography{MCVFAS}

\end{document}